\documentclass[11pt, reqno,amsmath,amsthm,amssymb,amscd]{amsart}
\usepackage{pifont}
\usepackage{amssymb}
\usepackage{mathrsfs,amssymb}
\usepackage{multicol}\multicolsep=0pt
\usepackage{pstricks,pst-node}
\usepackage[enableskew,vcentermath]{youngtab}

\usepackage[sort]{cite}

\def\crulefill{\leavevmode\leaders\hrule height 1pt\hfill\kern 0pt}
\long\def\QUERY#1{%
\leavevmode\newline%
\noindent$\star\star\star$\thinspace\textsf{Comment/Query}\crulefill\newline%
   \space #1\newline\hbox to 120mm{\crulefill}$\star\star\star$\newline}
\newtheorem{Theorem}{Theorem}[section]
\newtheorem{Lemma}[Theorem]{Lemma}
\newtheorem{Cor}[Theorem]{Corollary}
\newtheorem{Prop}[Theorem]{Proposition}

 \theoremstyle{definition}
\newtheorem{example}[Theorem]{Example}

\newtheorem{Defn}[Theorem]{Definition}

\newtheorem{rem}[Theorem]{Remark}

\numberwithin{equation}{section}
\theoremstyle{definition}

\makeatletter
\def\enumerate{\begingroup\ifnum\@enumdepth>3\@toodeep\else
      \advance\@enumdepth\@ne
      \edef\@enumctr{enum\romannumeral\the\@enumdepth}%
      \topsep\z@\parskip\z@
      \list{\csname label\@enumctr\endcsname}
        {\@nmbrlisttrue\let\@listctr\@enumctr
         \parsep\z@\itemsep\z@\topsep\z@
         \setcounter{\@enumctr}{0}
         \def\makelabel##1{\hss\llap{\rm ##1}}
       }\fi}

\makeatother

\let\bar=\overline
\let\epsilon=\varepsilon
\def\({\big(}
\def\){\big)}
\let\gdom\rhd

\def\0{\underline{0}}

\def\B{\mathscr B}

\def\G{\mathcal G}
\def\H{\mathscr H}

\DeclareMathOperator{\Rad}{Rad} \DeclareMathOperator*{\Res}{Res}

\let\gdom\rhd
\let\gedom\unrhd

\def\m{\mathfrak m}

\def\Std{\mathscr{T}^{std}}

\def\m{\mathfrak m}
\def\n{\mathfrak n}
\def\s{\mathfrak s}
\def\ts{\tilde\s}
\def\t{\mathfrak t}
\def\u{\mathfrak u}
\def\v{\mathfrak v}
\def\bfs{\s}
\def\bft{\t}
\def\F{\mathcal F}
\def\G{\mathcal G}
\def\Hom{\text{Hom}}
\def\Ind{\text{Ind}}
\def\Res{\text{Res}}

\def\ss{\mathsf s}
\def\ts{\mathsf t}

\def\res{\text{res} }

{\catcode`\|=\active
  \gdef\set#1{\mathinner{\lbrace\,{\mathcode`\|"8000%
                                   \let|\midvert #1}\,\rbrace}}
  \gdef\seT#1{\mathinner{\Big\lbrace\,{\mathcode`\|"8000%
                                   \let|\midverT #1}\,\Big\rbrace}}
}
\def\midvert{\egroup\mid\bgroup}
\def\midverT{\egroup\,\Big|\,\bgroup}

\def\Set[#1]#2|#3|{\Big\{\ #2\ \Big| \
           \vcenter{\hsize #1mm\centering #3}\Big\}}

\def\map#1#2{\,{:}\,#1\!\longrightarrow\!#2}

\def\la{{\lambda}}

\def\ssS{{\mathsf S}}

\def\bfT{{\mathbf T}}

\def\bfT{{\mathbf T}}

\begin{document}
\baselineskip15pt
\title{ The representations of  quantized walled Brauer algebras}
\author{ Hebing Rui and Linliang Song }
\address{H.R. Department of Mathematics,  East China Normal
University, Shanghai, 200062, China} \email{hbrui@math.ecnu.edu.cn}
\address{L.S. Department of Mathematics,  East China Normal
University, Shanghai, 200062, China}\email{52110601013@student.ecnu.edu.cn}
\thanks{H. Rui was supported partially  by NSFC in China and  Shanghai Municipal Science and Technology Commission
~11XD1402200.}

\date{\today}

\begin{abstract}
In this paper, we give a criterion on the semisimplicity of   quantized walled Brauer algebras
$\mathscr B_{r,s}$ and classify its simple modules  over an arbitrary field  $\kappa$.
\end{abstract}

\sloppy \maketitle
\section{Introduction}

Schur-Weyl  reciprocities set up close relationship between polynomial  representations of
general linear groups ${GL}_n$ over $\mathbb C$  and representations of  symmetric groups $\mathfrak S_r$~\cite{Gr}.
Such results have been generalized in various cases. In \cite{KM}, Kosuda and Murakami studied  mixed
Schur-Weyl duality between quantum general linear  group $U_q(\mathfrak{gl}_n)$ and quantized walled Brauer algebras $\mathscr B_{r, s}$
with single parameter over $\mathbb C$. Since then, quantized walled Brauer algebras
have been studied extensively in \cite{Le, DDS1, DDSII, Enyang2} etc.

A quantized walled Brauer algebra $\mathscr B_{r, s}$  with two parameters was  defined by Leduc in \cite{Le}. This is a cellular  algebra~\cite{Enyang2} over a commutative
ring  $R$ containing $1$.
 In fact, Enyang has shown that any cellular basis of the Hecke algebras associated
to symmetric groups can be lifted to a cellular basis of $\mathscr B_{r, s}$. In particular, using  anti-symmetrizers of Hecke algebras
instead of symmetrizers yields  a cellular basis of
$\mathscr B_{r,s}$ in Theorem~\ref{celb}. Our motivation for using this cellular basis  is
that bases of the corresponding  cell modules can be used to classify
 singular vectors or highest vectors appearing in the mixed tensor product of the natural module and its dual over
$U_q(\mathfrak{gl}_n)$. Such results can be used to determine  decomposition matrices of $\mathscr B_{r, s}$.
 Details will be given in \cite{Rsong}.

The aim of this paper is to study the representations of
over an arbitrary field $\kappa$ via the representation theory of cellular
algebras in \cite{GL}. In section~2, we recall the
definition of $\mathscr B_{r,s}$ and  list  some of its
properties. A cellular basis of $\mathscr B_{r,s}$ will be given in section~3.
We use  certain idempotents of
$\mathscr B_{r, s}$  to  construct Schur functors in section~4. We also  prove branching rule for cell modules of $\mathscr B_{r, s}$.
In section~5, we classify  irreducible $\mathscr B_{r,s}$-modules over $\kappa$. Finally,  we give a criterion on the semisimplicity of  $\mathscr B_{r, s}$
over $\kappa$. Such a result, which generalizes \cite[6.7]{KM}, can be considered as a counterpart of \cite[Theorem~6.3]{CVDM} for walled Brauer algebras.

\section{The quantized walled Brauer  algebra}
Throughout, we assume that  $R$ is the localization of $\mathbb Z[q, q^{-1}, \rho,
\rho^{-1}]$ at $q-q^{-1}$, which contains
 $\delta=(\rho-\rho^{-1})(q-q^{-1})^{-1}$.

\begin{Defn}\label{wbmw}~\cite{Le} Fix
$r,s \in \mathbb Z^{>0}$. The quantized walled Brauer algebra
${\mathscr{B}}_{r,s}$ is the associative $R$-algebra with generators
$e_1, g_i, g_j^*$,  $1\le i\le r-1$ and  $1\le j\le s-1$ subject
to the following relations
{\small
\begin{multicols}{2}
\begin{enumerate}
    \item $(g_i-q) (g_i+q^{-1})=0$,  $1\le i<r$,
\item $g_ig_j=g_jg_i$,  $|i-j|>1$, \item
$g_ig_{i+1}g_i=g_{i+1}g_ig_{i+1}$, $1\le i<r-1$,
\item $g_ie_1=e_1g_i$,  $i\neq 1$,
\item $e_1g_1e_1 =\rho e_1$, \item
$e_1^2=\delta e_1$, \item $g_i g^*_j=g^*_j g_i$,
\item $(g^*_i-q) (g^*_i+q^{-1})=0$,  $1\le i<s$,
\item $g^*_ig^*_j=g^*_jg^*_i$, $|i-j|>1$, \item
$g^*_ig^*_{i+1}g^*_i=g^*_{i+1}g^*_ig^*_{i+1}$, $1\le i<s-1$,
\item $g^*_ie_1=e_1g^*_i$, $i\neq 1$,\item $e_1g^*_1e_1=\rho e_1$, \item $e_1{g_{1}}^{-
1}g^*_{1} e_1g_{ 1} = {e_1}{g_{1}}^{- 1}g^*_{1 } {e_1}g^*_{1}$,\item
${g_{1}}{e_1}{g_{ 1}}^{ - 1}g^*_{1} e_1 = g^*_{1} e_1{g_{ 1}}^{ -
1}g^*_{1} {e_1}$.
\end{enumerate}
\end{multicols}}
\end{Defn}

\begin{rem} \label{rbigs} It follows from Definition~\ref{wbmw} that $\mathscr B_{r, s}\cong \mathscr B_{s, r}$.
When we discuss $\mathscr B_{r,s}$, we can  assume $r\ge s$ without
loss of any generality.
\end{rem}


If we allow $s=0$, then $\mathscr B_{r,s}$ is the usual Hecke
algebra $\H_{r}$ associated to the symmetric group $\mathfrak S_r$. More explicitly, it is generated by $g_i, 1\le
i\le r-1$,  subject to the defining relations (a)-(c) in Definition~\ref{wbmw}. In general,
$\mathscr B_{r,s}$ contains two subalgebras which are isomorphic to
$\H_r$ and $\H_s$, respectively.

\begin{Lemma}\cite{Enyang2} \label{inv} There is an $R$-linear
anti-involution $\sigma$ on $\mathscr{B}_{r,s}$ which fixes all
generators  $e_1,  g_i$ and $g_j^*$,  $1\le i\le r-1$ and $1\le j\le
s-1$. \end{Lemma}

For convenience, we write $g_{i,j}=g_{i-1}g_{i-2}\cdots g_j$, $i>j$,
$g_{i,i}=1$ and  $g_{i, j}=g_{i} g_{i+1}\cdots g_{j-1}$, $i<j$.
Similarly, we have the notation $g_{i,j}^*$. Given two positive integers $i, j$ with $
i\leq r$ and $ j\leq s$, let \begin{equation}\label{eij}  e_{ i, j} = g_{ 1, i }^{-1}
 g_{j, 1 }^*  e_1{g_{ 1, i}} ({g^*_{ j,1}})^{-1}\text{ and  $
 \bar e_{i,j} =g_{1,i}^{-1} g^*_{j, 1} e_1 g_{1,j}^*
g_{i,1}^{-1}$.}\end{equation}


\begin{Lemma}\label{tool} Let  $e_i=e_{i,i}$, $i\le \min\{r,s\}$. Then
\begin{enumerate}
\item $e_{i}{g_k} = {g_k}{e_{i}}$,  $ i < k < r$, and   $ {e_{i}}g_l^ *  = g_l^ * {e_{i}}$,  $i <l < s$,
\item $e_{i}^2  = \delta e_{i}$, $1\leq i\leq min\{r,s\}$,
\item $e_{ i}{g_{i }}^{\epsilon}{e_{ i}} =e_{ i}(g_{i }^*)^{\epsilon}{e_{i}} = \rho^{\epsilon} e_{ i}$,  $1\leq i<
\text{min}\{r,s\}$, $\epsilon\in \{1, -1\}$,
\item $e_{ i}{g_{ i }}(g_{ i }^*)^{-1}{e_{ i}} = {e_{ i}}g_{ i }^ * g_{i}^{ - 1}{e_{  i}} = {e_{i}}{e_{ i +1}}=e_{i+1}e_i $,  $1\leq i<\text{min}\{r,s\}$,
\item $e_{ i}{e_{i + 1}}{g_{ i}} = {e_{i+1}}{e_{i}}g_{i}^*$, $1\leq i<\text{min}\{r,s\}$,
\item $g_{i}{e_{i}}{e_{i+1}} = g_{i}^* {e_{i+1}}{e_{i}}$,  $1 \leq i<\text{min} \{r,s\}$,
\item $e_{i}{e_{j}} = {e_{j}}{e_{i}}$,  $1\leq i,j\leq \text{min}\{r,s\}$.\end{enumerate} \end{Lemma}

\begin{proof}  (a) follows from Definition~\ref{wbmw}(b),(d),(g),(i),(k).
By (\ref{eij}),  $$e_{i}^2 = g_{i-1}^{-1 }g_{i-1}^* e_{i-1}^2{g_{i-1}} (g_{i - 1}^*)^{
-1}=\delta g_{i-1}^{-1 }g_{i-1}^* e_{i-1}{g_{i-1}} (g_{i - 1}^*)^{
-1}=\delta e_i, $$
where the second equality follows from   induction assumption on $i-1$. This proves (b).   By Definition~\ref{wbmw}(g), braid relations, (a),  and induction assumption on $i-1$, we have   $$
{e_{i}}{g_{i}}{e_{i}}  = g_{i-1}^{ - 1}g_{i-1}^* g_i^{-1} {e_{i-1}}{g_{i-1}}{e_{i-1}} g_i{g_{i-1}}(g_{i - 1}^ *)^{-1}=\rho e_i. $$
One can check   $e_{ i}{g^*_{i }}{e_{ i}}=\rho e_{i}$, similarly.
Finally, the remaining cases of (c) follows from
 (b) and Definition~\ref{wbmw}(a) or (h).
 By (a) and  braid relations, we have
 $$ e_i g_i^{-1} g_i^* e_i g_i=g_{i+1,i-1}^{-1} g_{i-1, i+1}^* e_{i-1} g_{i-1}^{-1}g_{i-1}^* e_{i-1} g_{i-1} g_{i+1, i-1} ( g_{i-1, i+1}^*)^{-1}.$$
 Using induction assumption on $i-1$ together with (a) and braid relations, we have
 \begin{equation} \label{eii1} e_i g_i^{-1} g_i^* e_i g_i=e_i
g_i^{-1} g_i^*e_i g_i^*\end{equation}
By similar arguments, we have
\begin{equation}\label{eii2}  g_i e_i g_i^{-1} g_i^* e_i =g_i^* e_i g_i^{-1}
g_i^*e_i.\end{equation}
 So, $ {e_{ i}}g_{ i }^ * g_{i}^{ - 1}{e_{
i}}=e_ie_{i+1}=
e_{i+1}e_i$.  Further, via these equalities together with (b)-(c)  and Definition~\ref{wbmw}(a)(h), we have  $e_i(g_i^*)^{-1} g_{i}e_i = e_ie_{i +1}$.
This proves (d).   (e)-(f) follow from (d) and
(\ref{eii1})-(\ref{eii2}). Finally, we can assume $i<j$ without loss of any generality. By (a), (d), (e) and \eqref{eij},
$$\begin{aligned} e_i e_j&=e_i g_{i, j}^{-1} g_{j, i}^* e_i g_{i,j} (g_{j, i}^*)^{-1}= g_{i+1, j}^{-1} g_{j, i+1}^* e_i g_i^{-1}g_i^* e_ig_{i,j} (g_{j, i}^*)^{-1}\\
& = g_{i+1, j}^{-1} g_{j, i+1}^* e_i e_{i+1}g_{i,j} (g_{j, i}^*)^{-1}=g_{i+1, j}^{-1} g_{j, i+1}^* e_i e_{i+1}g_{i+1,j} (g_{j, i+1}^*)^{-1}=   e_j e_i.\\
\end{aligned}$$
 This proves (g).
\end{proof}


\begin{Prop}\label{ccen} Given  $r, s\in \mathbb Z^{>0}$, let
$$c_{r, s}=\sum_{i=1}^r \sum_{j=1}^s  \bar e_{i, j}-\rho^{-1} \sum_{i=2}^r \sum_{j=1}^{i-1} g_{j,i}^{-1} g_{i, j+1}^{-1}
-\rho \sum_{i=2}^s \sum_{j=1}^{i-1} g^*_{i,j} g^*_{j+1,i}.$$
Then  $c_{r, s}$ is central in $\mathscr B_{r, s}$.
\end{Prop}
\begin{proof} If  $r+s=2$,  then   $r=s=1$ and
$ c_{1,1}=e_1$, which is   central in $\mathscr B_{1,1}$. In the remaining part of the proof, we assume  $r+s\ge 3$.

We claim that $e_1$ commutes with $c_{r, s}$ by induction on $r+s$.  As mentioned in Remark~\ref{rbigs}, we assume
$r\ge s$. So,  $r\ge 2$ and
\begin{equation} \label{ind1} c_{r, s}-c_{r-1, s}=\sum_{j=1}^s \bar e_{r, j}-\rho^{-1} \sum_{j=1}^{r-1} g_{j, r}^{-1} g_{r, j+1}^{-1}.\end{equation}
By induction assumption on $\mathscr B_{r-1, s}$, it suffices to prove   $e_1 (c_{r, s}-c_{r-1, s})=(c_{r, s}-c_{r-1, s})e_1$.

In fact, by Lemma~\ref{tool}(a),(d),(e), $\bar e_{r, j} e_1 =e_1\bar e_{r, j}$,
$\forall j>1$. By Definition~\ref{wbmw}(d)-(e) and Lemma~\ref{tool}(c), $$\bar e_{r, 1} e_1-e_1\bar e_{r,1} =\rho^{-1}(
g_{1,r}^{-1} g_{r, 2}^{-1} e_1- e_1g_{1,r}^{-1} g_{r,2}^{-1}).$$
Finally, using   Definition~\ref{wbmw}(d) yields $$ e_1 \sum_{j=2}^{r-1} g_{j, r}^{-1} g_{r, j+1}^{-1}=\sum_{j=2}^{r-1} g_{j, r}^{-1} g_{r, j+1}^{-1} e_1.$$
So, $e_1 (c_{r, s}-c_{r-1, s})=(c_{r, s}-c_{r-1, s})e_1$, proving our claim.

Now,  we prove $c_{r, s} g_k=g_k c_{r, s}$ for all $k, 1\le k\le r-1$.  In fact, by  Definition~\ref{wbmw}(b)(c)(i)(j),
$\bar e_{k, j} g_{i}=g_i \bar  e_{k, j}$ for  $k\not\in \{i, i+1\}$.
   Since  $(\bar
e_{i+1, j} +\bar  e_{i,j})g_i=g_i (\bar e_{i+1, j} +\bar e_{i,j})$,  we have
 $$  \sum_{i=1}^r \sum_{j=1}^s \bar e_{i, j}
g_k=g_k\sum_{i=1}^r \sum_{j=1}^s  \bar e_{i, j}, \text{ for all $k, 1\le k\le r-1$}.$$

If we use $g_k$ instead of $g_k^{-1}, 1\le k\le r-1$ in   $x:=\sum_{i=2}^r
\sum_{j=1}^{i-1} g_{j,i}^{-1} g_{i, j+1}^{-1}$ , then $x$ is the summation
of the Murphy elements of $\H_r$, which is central in $\H_{r}$ (see e.g. \cite{Ma}). However,
 from Definition~\ref{wbmw}(a)-(c), one can see easily that $\H_r$ can be defined via $g_i^{-1}$.  So $x$  is a central in  $\H_r$. By Definition~\ref{wbmw}(g),
 $c_{r, s} g_k=g_k c_{r, s}$, $\forall k, 1\le k\le r-1$. Finally, one can check $c_{r, s}
g^*_{k}=g^*_k c_{r, s}$, similarly.
\end{proof}


It is known that $\mathfrak S_r$ is
 generated by $s_i$, the  basic transposition $(i, i+1)$,
$1\le i\le r-1$. Let  $\mathfrak S_{r}\times \mathfrak S_s$ be the
product of $\mathfrak S_{r}$ and $ \mathfrak S_s$. We use $s_i^*$ to
denote the basic transposition $(i, i+1)$ in $\mathfrak S_s$.

For convenience, we write $s_{i, j}=s_{i-1} s_{i-1,j}$, $i >j$, $s_{i,i}=1$ and $s_{i,j}=s_{i}s_{i+1, j}$,
$i<j$. Similarly, we have the notation $s^*_{i,j}$. The following result gives the explicit description on  $\mathscr{D}_{r,s}^f$  in \cite{Enyang2}.

\begin{Lemma}\label{rcs} Fix $r, s\in \mathbb Z^{>0}$ and $f\in \mathbb N$  with $f\le \min \{r, s\}$. Let $\mathfrak{G}_f$ be the subgroup of
$\mathfrak{S}_r\times\mathfrak{S}_s$ generated by $s_{i} s^*_{i} $, $1\le i\le f-1$. Then  $\mathscr{D}_{r,s}^f$ is  a complete set of
right coset representatives for
$\mathfrak{S}_{r-f}\times\mathfrak{G}_f\times\mathfrak{S}_{s-f}$ in
$\mathfrak{S}_r\times\mathfrak{S}_s$ where
\begin{equation}\label{rcs1}  \mathscr{D}_{r,s}^f=\{ s_{f,i_f} s^*_{f, j_f} \cdots
s_{1,i_1}s^*_{1,{j_1}}|  k \le {j_k},
 1 \le i_1< i_2 < \cdots <i_f\le r \}.\end{equation} \end{Lemma}

\begin{proof}We denote  by $\tilde {\mathscr{D}}_{r,s}^f$ the right-hand side of \eqref{rcs1}, and by $\mathscr{D}_{r,s}^f$ a complete set of
right coset representatives. Then obviously $\tilde {\mathscr{D}}_{r,s}^f\subset\mathscr{D}_{r,s}^f$.

In order to verify the inverse inclusion, it  suffices to prove that $|\tilde {\mathscr{D}}_{r,s}^f|$, the cardinality of $\tilde {\mathscr{D}}_{r,s}^f$, is
$\frac{r!s!}{(r - f)!(s - f)!f!}=C^{f}_rC^{f}_sf!$, which is clearly the cardinality of  $\mathscr{D}_{r,s}^f$, where $C^f_r$ is the binomial number.  This will be done by induction on $f$ as follows.

If  $f=0$, there is nothing to be proven.
Assume $f\ge1$. For any element in \eqref{rcs1}, we have
$i_f\ge f$. For each fixed $i:=i_f$, there are  $s-f+1$ choices of $j_f$ with $j_f\ge f$, and further, conditions for other indices are simply conditions for $\mathscr{D}_{i-1,s}^{f - 1}$.
So,
$$\begin{array}{llll}
|\tilde {\mathscr{D}}_{r,s}^f| \!\!\!&= (s \!-\! f\!+\! 1)\sum\limits_{i=f}^{r}|\mathscr{D}_{i-1,s}^{f - 1} |
\\[11pt]&=
(s\!-\!f\!+\!1)\sum\limits_{i=f}^{r}C_{i-1}^{f-1}C_s^{f-1}(f\!-\!1)!
=\sum\limits_{i=f}^{r}C_{i-1}^{f-1}C_s^ff!=C_r^fC_s^ff!,
\end{array}$$
where the second equality follows from induction assumption on $f$, and the last follows from the well-known combinatorics formula $C_r^i=C_{r-1}^i+C_{r-1}^{i-1}$.
\end{proof}

\begin{Lemma} \label{subiso1} Fix   $r, s, f\in \mathbb Z^{>0}$ with  $f\le \min\{r, s\}$.
 Let  ${\mathscr{B}}_{r,s}(f)$ be the subalgebra of
$\mathscr{B}_{r,s}$ generated by $e_{f+1}$,  $g_{i}$ and $g^*_j$,
$f+1\le i< r$ and $f+1\le j< s$. Then $
{\mathscr{B}}_{r,s}(f)\cong \mathscr{B}_{r - f,s - f}$.
\end{Lemma}
\begin{proof} The required
isomorphism  sends $e_{f+1}$, $g_{f+i}, g_{f+j}^*$ to $e_1$, $
g_{i}, g_{j}^*$, respectively. One can compare the defining relations in Definition~\ref{wbmw} and the equalities in
Lemma~\ref{tool}.
\end{proof}

 We denote  $\mathscr B_{r, s}(f)$ by $R$ if $r=s=f$.

\begin{Lemma}\label{fixef} Given a positive integer  $f$ with  $f\le \min\{r, s\}$, we have  $\sigma(e^f)=e^f$ where
$e^f=e_1e_2\cdots e_{f}$ and $\sigma$ is given in Lemma~\ref{inv}.
\end{Lemma}
\begin{proof} By Lemma~\ref{inv},  $\sigma(e^f)=e^f$ if $f=1$. In general,  by induction on $f$, we have $\sigma(e^f)=\sigma(e_f) e^{f-1}$. So, we need to prove  $\sigma(e_f) e^{f-1}= e^{f}$.
In fact, $$\begin{aligned} \sigma(e_f) e^{f-1}&=(g_{1, f}^*)^{-1} g_{f, 1} e_1 g_{1, f}^* g_{f, 1}^{-1} e_1\cdots e_{f-1}\\
&=(g_{1, f}^*)^{-1} g_{f, 1} e_1 g_1^*g_1^{-1}e_1 g_{2, f}^* g_{f, 2}^{-1} e_2\cdots e_{f-1}\\
&=(g_{1, f}^*)^{-1} g_{f, 1} e_1e_2  g_{2, f}^*  g_{f, 2}^{-1} e_2\cdots e_{f-1} \\
&=(g_{1, f}^*)^{-1} g_{f, 1} e_1e_2\cdots e_f=e^f.\\
\end{aligned}$$
We remark that the second, third and forth equalities follow from Lemma~\ref{tool}(a) and (d), and the last equality follows from Lemma~\ref{tool}(f)-(g).
\end{proof}

We define
  \begin{equation}\label{gd}g_d = {g_{f,i_f}} g^*_{f ,j_f}  \cdots
{g_{1, i_1}}g^*_{1,j_1}.\end{equation} for  each  $d\in \mathscr D_{r, s}^f$ if   $d=
s_{f,i_f} s^*_{f,j_f} \cdots s_{1, i_1}s^*_{1,{j_1}}$ with $j_k\ge
k$ and $i_{\ell}< i_{\ell+1}$.

The following result is motivated by Yu's work on cyclotomic Birman-Murakami-Wenzl algebras in \cite{Yu}.

\begin{Prop}\label{tool1} Fix $r, s, f\in \mathbb Z^{>0}$ with  $ f\le \min\{r, s\}$. Let $N_f$ be
the left  ${\mathscr{B}}_{r,s}(f)$-module  generated
by $\bar V_{r, s}^f=\{ e^f{g_d}\mid d \in \mathscr {D}_{r,s}^f\}$, where $g_d$'s  are  given in (\ref{gd}).
Then $N_f$ is a right ${\mathscr{B}}_{r,s}$-module.
\end{Prop}

\begin{proof}  We claim  $M_f$ is a right $\mathcal
{{\mathscr{B}}}_{r,s}$-module, where  $M_f$ is the left
${\mathscr{B}}_{r, s}(f)$-module generated by
$V_{r,s}^f=\{ e^f{g_{f,i_f}}g^*_{f,j_f} \cdots {g_{1,
i_1}}g^*_{1,j_1}\mid i_k, j_k \ge k, 1\le k\le f\}$.

First, we assume  $f=1$. By  Definition~\ref{wbmw}(h), and
Lemma~\ref{tool}(a), (c)-(d), we have $$ e_1
{g_{1,i_1}}g_{1,j_1}^*e_1=(e_2e_1+\rho(q-q^{-1})  e_1)g_{2,
i_1}g_{2, j_1}^* \text{ for  $i_1>1$  and  $j_1>1$.} $$
 Also, we have $e_1{g_{1,i_1}}{e_1}=\rho e_1 g_{2, i_1}$ and
$e_1g_{1,j_1}^*{e_1}=\rho e_1  g_{2, j_1}^*$.
In any case, $M_1$ is stable under the action of $e_1$.
By Definition~\ref{wbmw}, it is easy to  verify that $M_1$ is stable
under the actions of $g_i$'s and $g^*_j$'s. So, $M_1$ is a right
$\mathcal {{\mathscr{B}}}_{r,s}$-module, proving our result for
$f=1$.  Using the result for $f=1$ repeatedly yields the result for
general $f$.

By definition,  $N_f\subseteq  M_f$.  So, our result follows if  $M_f\subseteq
N_f$. We prove it by induction on $f$. The case $f=1$ is trivial since
 $M_1=N_1$. In general, by induction
assumptions on both $f-1$ and $f=1$, we have
$$ {\mathscr{B}}_{r,s}(f) V_{r,s}^f \subseteq
\sum_ {i_f, j_f\ge f}
{\mathscr{B}}_{r,s}(f) e_fg_{f,i_f}g_{f,j_f}^*\bar{V}_{r,s}^{f-1}.$$
So, $M_f\subseteq N_f$ if  $e_fg_{f,i_f}g_{f,j_f}^* e^{f-1} g_d\in
N_f$, for any $e^{f-1} g_d \in \bar{V}_{r,s}^{f-1}$.

Write $g_d=g_{f-1, i_{f-1}} g_{f-1, j_{f-1}}^*\cdots g_{1,
 i_1}g_{1, j_1}^*$ with $i_1<\cdots <i_{f-1}$. If $i_f>i_{f-1}$, there is nothing to be proved.
 So, we assume  $i_f\leq i_{f-1}$. By Lemma~\ref{tool}(e) and Definition~\ref{wbmw},
we have  \begin{equation} \label{gf}
e_fg_{f,i_f}g_{f,j_f}^*e_{f-1}g_{f-1,i_{f-1}}g_{f-1,j_{f-1}}^*
=e_fe_{f-1}g_{f,i_{f-1}}g_{f-1,i_f-1}g_{f-1,j_f}^*g_{f-1,j_{f-1}}^*,\end{equation}
and \begin{equation} \label{gf1}
g_{f-1,j_f}^*g_{f-1,j_{f-1}}^*=\begin{cases}
g_{f,j_{f-1}}^*g_{f-1,j_f-1}^*+(q-q^{-1})g_{f,j_f}^*g_{f-1,j_{f-1}}^*, &\text{ if $j_{f-1}\ge j_f,$}\\
g^*_{f,j_{f-1}+1}g^*_{f-1,j_f}, &\text{ if $j_{f-1}< j_f$}.
\\ \end{cases}\end{equation}
Applying the previous arguments repeatedly yields    $e_fg_{f,i_f}g_{f,j_f}^* e^{f-1} g_d\in \bar V_{r, s}^f$.\end{proof}

The following result can be considered as the left version of Proposition~\ref{tool1}.

\begin{Cor}\label{tool2} Fix $r, s, f\in \mathbb Z^{>0}$ with $ f\le \min\{r, s\}$.  Let $N_f$ be the
right $ {\mathscr{B}}_{r,s}(f)$-module  generated by
$\{\sigma(g_d) e^f\mid d \in\mathscr {D}_{r,s}^f\}$. Then $N_f$ is a
left $ {\mathscr{B}}_{r,s}$-module.
\end{Cor}

\begin{Lemma}\label{quo1}
Let $I$ be the two-sided ideal of ${\mathscr{B}}_{r,s}(f)$
generated by $e_{f+1}$. Then ${\mathscr{B}}_{r,s}(f)/I\cong  \H_{r-f}\otimes  \H_{s-f}$. \end{Lemma}
\begin{proof}Straightforward verification.\end{proof}

\section{A cellular basis of $\mathscr B_{r, s}$}
The aim of this section is to give  a cellular basis of $\mathscr B_{r, s}$ in  Theorem~\ref{celb}.
We remark that  Enyang~\cite{Enyang2} has shown that arbitrary cellular bases for Hecke algebras associated to symmetric groups can be lifted
to cellular bases of the quantized walled Brauer algebras. In this sense, the cellular basis of $\mathscr B_{r, s}$ given in  Theorem~\ref{celb} can be obtained from~\cite[Theorem~6.13]{Enyang2}.


\begin{Defn}\cite{GL}\label{GL}
    Let  $A$ be  an $R$--algebra, where  $R$ is a commutative ring
containing the multiplicative identity $1$.
    Fix a partially ordered set $\Lambda=(\Lambda,\gedom)$ and for each
    $\lambda\in\Lambda$ let $T(\lambda)$ be a finite set. Finally,
    fix $C_{\s\t}\in A$ for all
    $\lambda\in\Lambda$ and $\s,\t\in T(\lambda)$.

    Then the triple $(\Lambda,T,C)$ is a \textsf{cell datum} for $A$ if:
    \begin{enumerate}
    \item $\set{C_{\s\t}|\lambda\in\Lambda\text{ and }\s,\t\in
        T(\lambda)}$ is an $R$--basis for $A$;
    \item the $R$--linear map $*\map AA$ determined by
        $(C_{\s\t})^*=C_{\t\s}$, for all
        $\lambda\in\Lambda$ and all $\s,\t\in T(\lambda)$ is an
        anti--involution of $A$;
    \item for all $\lambda\in\Lambda$, $\s\in T(\lambda)$ and $a\in A$
        there exist scalars $r_{\t\u}(a)\in R$ such that
        $$C_{\s\t} a
            =\sum_{\u\in T(\lambda)}r_{\t\u}(a)C_{\s\u}
                     \pmod{A^{\gdom\lambda}},$$
            where
    $A^{\gdom\lambda}=R\text{--span}%
      \set{C_{\u\v}|\mu\gdom\lambda\text{ and }\u,\v\in T(\mu)}$.
     Furthermore, each scalar $r_{\t\u}(a)$ is independent of $\s$. \end{enumerate}
     An algebra $A$ is a \textsf{cellular algebra} if it has
    a cell datum. We  call
    $\set{C_{\s\t}|\s,\t\in T(\lambda), \lambda\in\Lambda}$
    a \textsf{cellular basis} of $A$.
\end{Defn}

Unless otherwise stated, we always  consider right $A$-modules. Via anti-involution in Definition~\ref{GL}, all right $A$-modules can be considered as left modules.
For each
$\lambda\in\Lambda$ fix $\t \in T(\lambda)$ and let $C_{\s}
       =C_{\t\s}+ A^{\rhd \lambda}$.
       The right cell module $C(\lambda)$ of $A$ with respect to $\lambda\in \Lambda$ can be
       considered as  the free $R$--modules with basis $\set{C_{\s}|\s\in
T(\lambda)}$. Further, for any $a\in A$,  $$ C_{\s} \cdot a
=\sum_{\u\in T(\lambda)}r_{\s\u}(a)C_{\u}
                     $$ where the scalars   $r_{\s\u}(a)$ are
                     determined by Definition~\ref{GL}(c).
Similarly, we have the left cell modules of $A$.

 Before we construct a
cellular basis of $\mathscr B_{r, s}$, we need the Murphy basis for
$\H_n$, which is a cellular basis in the sense of \cite{GL}. First,
we recall some combinatorics.

A composition $\lambda$ of $n$ with at most  $d$ parts is a sequence
of non--negative integers $\lambda=(\lambda_1,\lambda_2,\dots,
\lambda_d)$ such that $|\lambda|:=\sum_{i=1}^d\lambda_i=n$.  If
$\lambda_i\ge \lambda_{i+1}$,  $ 1\le i\le d-1$, then $\lambda$ is
called a partition of $n$ with at most $d$ parts. Let $\Lambda(d,
n)$ (resp. $\Lambda^+(d, n)$)  be the set of all compositions (resp.
partitions)  of $n$ with at most $d$ parts.  We also use
$\Lambda^+(n)$ to denote the set of all partitions of $n$. It is
known that $\Lambda^+(d, n)$ is the poset with dominance order
$\trianglelefteq$ as the partial order on it. More explicitly, $
\lambda\trianglelefteq \mu$  for $\lambda, \mu\in \Lambda^+(d, n)$
if $ \sum_{j=1}^i \lambda_j\le \sum_{j=1}^i \mu_j$ for all possible
$i\le d$. Write $\lambda\vartriangleleft \mu$ if
$\lambda\trianglelefteq \mu$ and $\lambda\ne \mu$.

Let $\lambda=(\lambda_1, \lambda_2...)\in \Lambda^+(n)$.  The Young diagram $[\lambda]$ is a
collection of boxes (or nodes) arranged in left-justified rows with $\lambda_i$
boxes in the $i$-th row of $[\lambda]$. We use $(i, j)$ to denote the box $p$ if  $p$ is in $i$-th row and $j$-th column.
A box $(i, \lambda_i)$ (resp.,  $(i, \lambda_i+1)$)  is called a removable (resp., addable ) node of $\lambda$ (or $[\lambda]$)
 if $\lambda_{i}-1\ge \lambda_{i+1}$ (resp. $\lambda_{i-1}\ge \lambda_{i}+1$).   Let $\mathscr R(\lambda)$ (resp., $\mathscr A(\lambda)$) be the
 set of all removable (resp., addable ) boxes of $\lambda$.

 A $\lambda$-tableau $\s$ is
obtained by inserting $i, 1\le i\le n$ into $[\lambda]$ without
repetition. A $\lambda$-tableau $\s$ is said to be standard if the
entries in $\s$ are increasing both from left to right in each row
and from top to bottom in each column. Let $\Std(\lambda)$ be the
set of all standard $\lambda$-tableaux.

The symmetric group  $\mathfrak S_n $  acts on a $\lambda$-tableau
$\s$ by permuting its entries. Let $\t^\lambda$ (resp.
$\t_{\lambda}$) be the $\lambda$-tableau obtained from the Young
diagram $[\lambda]$ by adding $1, 2, \cdots, n$ from left to right
along the rows (resp. from top to bottom along the columns). For
example, if $\lambda=(4,3,1)$, then
\begin{equation}\label{tla}
\t^{\lambda}=\young(1234,567,8), \quad \text{ and }
\t_{\lambda}=\young(1468,257,3).\end{equation}
We write $w=d(\s)$ if $\t^\lambda w=\s$. Then  $d(\s)$ is
uniquely determined by $\s$.

Let $\mathcal Z=\mathbb Z[q, q^{-1}]$.   It is known that $\{g_w\mid
w\in \mathfrak S_n\}$ is a $\mathcal Z$--basis of $\H_n$, where
$g_w=g_{i_1} \cdots g_{i_k}$ if $w=s_{i_1}\cdots s_{i_k}$ with
minimal $k$ which is called the length of $w$. Such an expression is
called a reduced expression of $w$. Further, it is well known that
$g_w$ is independent of a reduced expression of $w$.

Given a $\lambda\in \Lambda^+(n)$, let $\mathfrak S_\lambda$ be the
row stabilizer of $\t^\lambda$. Then $\mathfrak S_\lambda$ is the
Young  subgroup of $\mathfrak S_n$ with respect to $\lambda$. Let
\begin{equation}\label{xyla}
\m_\lambda=\sum_{w\in \mathfrak S_\lambda} q^{\ell(w)} g_w, \text{
and } \n_\lambda=\sum_{w\in \mathfrak S_\lambda} (-q)^{-\ell(w)}
g_w,\end{equation} where $\ell(w)$ is the length of $w$. It is well
known that
 \begin{equation}\label{sgnindex} \m_\lambda g_i=q \m_\lambda,
\quad \text{and}\quad  \n_\lambda g_i=-q^{-1} \n_\lambda, \forall
s_i\in \mathfrak S_\lambda\end{equation}

For any $\lambda\in \Lambda^+(n)$, the classical Specht module
$S_\lambda$ is $\m_\lambda g_{d(\t_\lambda)} \n_{\lambda'} \H_n$
where $\lambda'$ is the conjugate of $\lambda$. The following
result  is well known.

\begin{Prop}\cite{DJ1} \label{classsp} Suppose $\lambda\in \Lambda^+(n)$. Then $\{\m_\lambda g_{d({\t_\lambda})} \n_{\lambda'} g_{d(\t)}\mid \t\in
\Std(\lambda')\}$ is a $\mathcal Z$-basis of  $S_\lambda$.
\end{Prop}

In \cite{Mur}, Murphy constructed a $\mathcal Z$-basis of $\H_n$,
called Murphy basis. It is a cellular basis of $\H_n$ over $\mathcal
Z$. In the current paper, we use $\n_\lambda$ instead of
$\m_\lambda$ in his construction. The following result follows from Murphy's work in \cite{Mur}.

\begin{Theorem}\cite{Mur}\label{cellh} Let  $\H_n$  be defined over $\mathcal Z$. Then
$\{\n_{\s\t}\mid \s, \t\in \Std(\lambda), \lambda\in \Lambda^+(n)\}$
is a cellular basis of $\H_n$  where $\n_{\s \t}=g_{d(\s)^{-1}}
\n_\lambda g_{d(\t)}$.
\end{Theorem}

 For
each $\lambda\in \Lambda^+(n)$, let $C(\lambda)$ be the cell
module of $\H_n$ with respect to this cellular basis. The following
result is well known.

\begin{Prop}\label{classspe} For each $\lambda\in \Lambda^+(n)$, $S_\lambda\cong
C(\lambda')$ where $\lambda'$ is the conjugate of
$\lambda$.\end{Prop}

We begin to construct a cellular basis of $\mathscr B_{r, s}$. Fix
$r, s\in \mathbb Z^{>0}$. Let
\begin{equation}\label{poset}
 \Lambda_{r,s} = \left\{ (f,\lambda )| \lambda \in \Lambda_{r, s}^f, 0\le f \le \min \{ r,s \} \right\},\end{equation}
 where $\Lambda_{r,s}^f =\Lambda^+(r-f)\times \Lambda^+(s-f)$. So,
 each $\lambda\in \Lambda_{r,s}^f$ is of form $(\lambda^{(1)},
 \lambda^{(2)})$. We say that $(f,\lambda)\unrhd (\ell, \mu)$ if either $f>\ell$ or
$f=\ell$ and $\lambda\unrhd \mu$ in the sense
$\lambda^{(i)}\unrhd\mu^{(i)}$, $i=1,2$.  We write
 $(f,\lambda)\rhd (\ell,\mu)$ if $(f,\lambda)\unrhd (\ell,\mu)$
 and  $(f,\lambda)\neq(\ell,\mu)$.
 Then  $\Lambda_{r,s}$ is a poset.

Given a $\lambda\in \Lambda_{r, s}^f$, we define
$\t^\lambda=(\t^{\lambda^{(1)}}, \t^{\lambda^{(2)}})$ where
$\t^{\lambda^{(1)}}$ and $\t^{\lambda^{(2)}}$  are defined similarly
as (\ref{tla}). The only difference is that we have to use $f+i$
instead of $i$ in (\ref{tla}). Similarly, we have $\t_\lambda$.

\begin{example} Suppose $(r, s)=(2, 7)$, $f=1$ and  $(\lambda^{(1)},\lambda^{(2)}) =((1), (3,2,1))$.  We have
\begin{equation}\label{tlar} \t^{\lambda}=\left (\young(2), \quad \young(234,56,7)\right)
\quad \text{and} \quad \t_{\lambda}=\left (\young(2), \quad \young(257,36,4)\right).\end{equation}
\end{example}
For each $\lambda\in \Lambda^f_{r, s}$, let $\Std(\lambda^{(i)})$ be
the set of standard $\lambda^{(i)}$-tableaux which are obtained from
usual standard tableaux by using $f+j$ instead of $j$. Let
$\Std(\lambda)=\Std(\lambda^{(1)})\times \Std(\lambda^{(2)})$. If
$\s, \t\in \Std(\lambda)$ with $\s=(\s_1, \s_2)$ and $\t=(\t_1,
\t_2)$, we define
$$
\n_{\s\t}=\sigma(g_{d(\s_1)} g^*_{d(\s_2)})\n_{\lambda^{(1)}}
\n_{\lambda^{(2)}} g_{d(\t_1)}  g^*_{d(\t_2)},
$$
where $\sigma$ is the one given in Lemma~\ref{inv}. Let
$\mathscr{B}_{r,s}^{f}$  be the two-sided ideal of
$\mathscr{B}_{r,s}$ generated by $e^f$. Let
$\mathscr{B}_{r,s}^{\unrhd(f,\lambda)}$ be the two sided ideal of
$\mathscr{B}_{r,s}$ generated by $\mathscr{B}_{r,s}^{f+1}$ and all
$e^f \n_{\s \t}$ with $\s,\t\in \Std(\mu)$ and $(f,\mu) \unrhd (f,
\lambda)$. Define
$$\mathscr{B}_{r,s}^{\rhd(f,\lambda)}=\sum_{(f, \mu)\rhd (f, \lambda)}\mathscr{B}_{r,s}^{\unrhd(f,\mu)}.$$

The following  result follows from Proposition~\ref{tool1}
 and Lemma~\ref{quo1}, immediately.

\begin{Prop}\label{cell-m} Suppose  $(f, \lambda)\in \Lambda_{r, s}$. We have
\begin{enumerate} \item  $\Delta^R(f,\lambda)$ is a right
$\mathscr {B}_{r,s}$-module if $\Delta^R(f,\lambda)$  is the $R$-submodule of
$\mathscr{B}_{r,s}^{\unrhd(f,\lambda)}/~\mathscr{B}_{r,s}^{\rhd(f,\lambda)}$
spanned by $\{e^f\n_{\ts^\lambda \ss} g_d+
\mathscr{B}_{r,s}^{\rhd(f,\lambda)}|(\s, d) \in \Std(\lambda) \times
\mathscr{D}_{r,s}^f\}$;
\item $ \Delta^L(f,\lambda)$ is a left
$\mathscr {B}_{r,s}$-module if  $\Delta^L(f,\lambda)$  is the $R$-submodule of
$\mathscr{B}_{r,s}^{\unrhd(f,\lambda)}/~\mathscr{B}_{r,s}^{\rhd(f,\lambda)}$
spanned by $\{\sigma(g_d) e^f\n_{\ss\ts^\lambda } +
\mathscr{B}_{r,s}^{\rhd(f,\lambda)}|(\s, d) \in \Std(\lambda) \times
\mathscr{D}_{r,s}^f\}$.
\end{enumerate}
\end{Prop}

For each $(f, \lambda)\in \Lambda_{r, s}$, we define  $I(f,
\lambda)=\Std(\lambda)\times \mathscr D^f_{r, s}$. For any $(\s
,e), (\t, d)\in I(f, \lambda)$, we define
\begin{equation}\label{cellbasis} C_{(\s, e)(\t, d)}=\sigma(g_e)
e^f\n_{\s\t}g_d.\end{equation}

As we explained before, the following result can be obtained from \cite[Theorem~6.13]{ Enyang2}.

\begin{Theorem} \label{celb}Let $\mathscr
{B}_{r,s}$ be the quantized wall Brauer algebra over $R$. Then
$\mathcal C$ is a cellular $R$-basis of $\mathscr {B}_{r,s}$ over
the poset $\Lambda_{r, s}$, where $$\mathcal C=\cup_{(f, \lambda)\in
\Lambda_{r,s}} \{C_{(\s, e)(\t, d)}\mid (\s, e), (\t, d) \in I(f,
\lambda)\}.$$  The required anti-involution $\sigma$ is the one
given in Lemma~\ref{inv}.\end{Theorem}

\begin{proof}  Suppose $0\le f\le
\min\{r, s\}$. By Proposition~\ref{cell-m},  $\mathscr{B}_{r,s}^{f}
/\mathscr{B}_{r,s}^{f+1}$ is spanned by $C_{(\s, e)(\t, d)}+
\mathscr{B}_{r,s}^{f+1}$, for all  $(\s, e), (\t, d) \in I(f,
\lambda)$  and $ \lambda\in \Lambda_{r, s}^f$. So, $\mathscr B_{r,
s}$ is spanned by $\mathcal C$.  Counting the dimension of the
walled Brauer algebra $B_{r, s}$ in \cite{CVDM} yields the equality
$\# \mathcal C=(r+s)!$. So, $\mathcal C$ is $R$-linear independence.
Finally, by Proposition~\ref{cell-m} and Lemma~\ref{inv}, $\mathcal C$ is a cellular
basis in the sense of \cite{GL}.
\end{proof}

For each $(f, \lambda)$, we use $C(f, \lambda)$ to denote the
right  cell module of $\mathscr B_{r, s}$ with respect to the cellular
basis in Theorem~\ref{celb}. We denote $\lambda'$ by $(\mu^{(1)},\mu^{(2)})$ where $\mu^{(i)}$ is the conjugate of $\lambda^{(i)}$, $i=1, 2$.

\begin{Prop} \label{classcell} For each $(f, \lambda)\in \Lambda_{r, s}$, let $\tilde C(f, \lambda):=e^f \m_{\lambda'}g_{d(\t_{\lambda'})}
 \n_\lambda \mathscr B_{r,s} \pmod{ \mathscr \B_{r,s}^{f+1}}$. As right $\mathscr B_{r, s}$-module,   $C(f, \lambda)\cong \tilde C(f, \lambda)$.
\end{Prop}
\begin{proof}  By Propositions~\ref{tool1},   $e^f \m_{\lambda'}g_{d(\t_{\lambda'})}
 \n_\lambda\mathscr B_{r,s} $ is spanned by $ e^f \m_{\lambda'}g_{d(\t_{\lambda'})}
 \n_\lambda \mathscr B_{r, s}(f) g_d $ for all $d\in \mathscr D_{r, s}^f$.
 By Lemma~\ref{quo1}, we can use $\mathscr H_{r-f}\otimes \mathscr H_{s-f}$ instead of  $\mathscr B_{r, s}(f)$
 in  $\tilde C(f, \lambda)$. Using Proposition~\ref{classsp} yields a basis of  $e^f \m_{\lambda'}g_{d(\t_{\lambda'})}
 \n_\lambda \mathscr B_{r,s} \pmod{ \mathscr \B_{r,s}^{f+1}}$. Now, required isomorphism follows from
Proposition~\ref{classspe}, immediately.\end{proof}

\section{Inductions and Restrictions }
In this section,  unless otherwise stated, we always consider $\mathscr B_{r,s}$ over a field $\kappa$. We will   describe   certain restrictions and inductions  of the cell modules
of $\mathscr B_{r, s}$. This is motivated by Doran, Wales Hanlon's work on Brauer algebras over $\mathbb C$ in \cite{DWH}.

\begin{Lemma}\label{efe} Let $\mathscr B_{r, s}$ be over
$\kappa$. We have
\begin{enumerate} \item  $e_1\mathscr{B}_{r,s}e_1=
{\mathscr{B}}_{r,s}(1)e_1$.
\item If $s\ge 2$, then $\tilde e_{12} \mathscr{B}_{r,s}\tilde e_{1,2}= {\mathscr{B}}_{r,s}(1)\tilde
e_{12} $ and $(\tilde e_{1,2})^2 =\tilde e_{1,2}$ where $\tilde e_{1,2}=\rho^{-1} e_1 g_1^*$.
\item If $r\ge 2$, then  $ {f_{2,1}} \mathscr{B}_{r,s} {f_{2,1}}={\mathscr{B}}_{r,s}(1)
{f_{21}} $ and  $({f_{2,1}})^2 ={f_{2,1}}$  where ${f_{2,1}}=\rho^{-1} e_1 g_1$.
\end{enumerate}
\end{Lemma}
\begin{proof}  If $r+s=2$, then $r=s=1$ and $\mathscr B_{1, 1} (1)=\kappa$. In this case, we have (a) by $e_1^2=\delta e_1$.
Suppose $r+s\ge 3$. By Remark~\ref{rbigs}, we can assume $r\ge 2$. Then   $e_1=\rho^{-1} e_1 g_1 e_1$. We have
${\mathscr{B}}_{r,s}(1)e_1=e_1{\mathscr{B}}_{r,s}(1)  g_1 e_1\subseteq e_1\mathscr{B}_{r,s}e_1$.
 By
Proposition~\ref{tool1} for $f=1$, each element in $e_1\mathscr{B}_{r,s}e_1$
can be written as a linear combination of elements in
${\mathscr{B}}_{r,s}(1)e_1g_de_1$ with
$d\in\mathscr{D}_{r,s}^1$. Note that $g_d= g_{1,i}g_{1,j}^*$ for
some positive integers $i, j$. By Definition~\ref{wbmw}(d)(k), we
need only to deal with the case $i, j\in \{1, 2\}$.

 If $\{i, j\}\cap
\{1\}\neq \emptyset$, by Definition~\ref{wbmw}(e) or (f) or (l), $e_1 g_d e_1\in \mathscr{B}_{r,s}(1)e_1$.
Otherwise, by Definition~\ref{wbmw}(a)(l) and Lemma~\ref{tool}(d), we have
$$e_1 g_d e_1=e_1g_{1}g_{1}^*e_1=e_1 ( g_1^{-1} +(q-q^{-1}))  g_1^* e_1=e_1e_2+\rho(q-q^{-1})
e_1\in  \mathscr{B}_{r,s}(1)   e_1.$$ So,
$ \mathscr{B}_{r,s}(1) e_1\supseteq e_1\mathscr{B}_{r,s}e_1$, and (a) follows.
(b)-(c) follow from (a), immediately.
\end{proof}

\begin{Cor} \label{leftright} As algebras over  $\kappa$, we have
$ \mathscr{B}_{r,s}(1) \tilde
e_{12}\cong  \mathscr{B}_{r,s}(1)\cong
 \mathscr{B}_{r,s}(1) {f_{2,1}}$.
\end{Cor}

\begin{proof} Using cellular bases for $\mathscr B_{r, s}$ and $\mathscr{B}_{r,s}(1) $ yields bases for  $\mathscr{B}_{r,s}(1) \tilde
e_{12}$ and $\mathscr{B}_{r,s}(1) {f_{2,1}}$. In particular, we have
$$ \dim_\kappa \mathscr{B}_{r,s}(1) \tilde
e_{12}=\dim_\kappa  \mathscr{B}_{r,s}(1)=\dim_\kappa\mathscr{B}_{r,s}(1) {f_{2,1}}.$$   So, the homomorphism from  $ \mathscr{B}_{r,s}(1)$ to
$ \mathscr{B}_{r,s}(1) \tilde e_{12}$ (resp.  $\mathscr{B}_{r,s}(1) {f_{2,1}}$ ) sending $x$ to $x \tilde e_{1,2}$ (resp. $x  f_{2,1}$)  is the required algebra  isomorphism.\end{proof}

For any finite dimensional algebra  $A$ over $\kappa$, let $
\text{$A$-mod}$ be the category of left $A$-modules. Since we are considering $\mathscr B_{r, s}$, each left  $\mathscr B_{r, s}$-module  can be considered as a
right  $\mathscr B_{r, s}$-module via anti-involution $\sigma$ in Lemma~\ref{inv}.

 In the remaining part of this paper, we use $\mathfrak e_{r, s} $ to
denote either $\tilde e_{1, 2}$ or $f_{2, 1}$ in Lemma~\ref{eequ}.  By Lemma~\ref{efe}(b)-(c),
Corollary~\ref{leftright}  and standard arguments in \cite[Sect.
6]{Gr}, we have  the Schur functor $\F_{r,s}$ and the functor
$\G_{r,s}$
$$\begin{aligned} & \F_{r,s}:
 \mathscr{B}_{r,s}\text{-mod}\longrightarrow\mathscr{B}_{r,s}(1)\text{-mod},\\
& \G_{r, s}: \mathscr{B}_{r,s}(1)\text{-mod}
\longrightarrow \mathscr{B}_{r,s}\text{-mod},\\
\end{aligned} $$
such that for any left $\mathscr B_{r, s}$-module $M$ and any left $
{\mathscr B}_{r,s}(1)$-module  $N$,
$$\F_{r,s}(M)=\mathfrak e_{r, s}  M, \quad \text{and}\quad  \G_{r,s}(N)= \mathscr{B}_{r,s}
\mathfrak e_{r,s} \otimes_{\mathscr{B}_{r,s}(1)} N.$$
We remark that the right (resp. left) action of ${\mathscr{B}}_{r,s}(1)$ on  $\mathscr{B}_{r,s}
\mathfrak e_{r,s}$ (resp.  $\mathfrak e_{r, s}  M$) is given by $$(\mathscr{B}_{r,s}
\mathfrak e_{r,s}) \circ h=\mathscr{B}_{r,s}
\mathfrak e_{r,s} h\mathfrak  e_{r,s}=\mathscr{B}_{r,s} h
\mathfrak e_{r,s}$$ (resp. $h\circ \mathfrak e_{r, s}  M= \mathfrak e_{r, s}h  \mathfrak e_{r, s} M=h\mathfrak e_{r, s}  M$) for any $h\in {\mathscr{B}}_{r,s}(1) $.

For the simplification of notation, we use $\F$, $\G$ and $\mathfrak
e$  instead of $\F_{r, s}$, $\G_{r, s}$ and $\mathfrak e_{r, s}$. We
also use $\Hom$ instead of $\Hom_{\mathscr B_{r, s}}$ if there is no
confusion. By Lemma~\ref{subiso1},  $ {\mathscr{B}}_{r+1,s+1}(1)\cong \mathscr B_{r,
s}$.  By abuse of notation, we will use the same notation to denote
the cellular basis, cell modules et al.  for  $
{\mathscr{B}}_{r+1,s+1}(1)$.

Unlike what we did before, we consider the left cell modules in the remaining part of this section. As we mentioned before, left $\mathscr B_{r, s}$-modules can also
 be considered as right $\mathscr B_{r, s}$-modules. We remark that
Lemma~\ref{functors} for   walled Brauer algebras has been given  in \cite{CVDM}.

\begin{Lemma}\label{functors} Suppose that
$(f,\lambda)\in \Lambda_{r,s}$ and $(\ell,\mu)\in
\Lambda_{r+1,s+1}$. \begin{enumerate} \item $\F \G=1$,
\item  $\F(C (f,\lambda))\cong C(f-1,\lambda)$,
\item $\G(C(f,\lambda))\cong C(f+1,
\lambda)$,
\item
$\Hom (\mathscr{B}_{r+1,s+1} \mathfrak e,  C (\ell,\mu))\cong
\mathfrak  e C(\ell,\mu) $,
\item $\Hom(\G(C(f,\lambda)),C(\ell,\mu))\cong
\Hom( C(f,\lambda), \F
  (C(\ell,\mu)))$.
\end{enumerate}\end{Lemma}

\begin{proof} (a) follows from Lemma~\ref{efe} and Corollary~\ref{leftright}.
We prove  (b) under the assumption $f\ge 1$. Otherwise, the result is trivial
 since $\mathfrak e C(0, \lambda)=0$ and $C(-1,\lambda):=0$.

 By Lemma~\ref{efe}(a),
$\mathfrak e C(f, \lambda) $ has a basis $ e_1
\sigma(g_d) (e_2\cdots e_f)  \n_{\t\t^\lambda} + \mathscr B_{r, s}^{\rhd (f, \lambda)} $
where $d\in \mathscr D^{f-1}_{ r, s}$ and $\t\in \Std(\lambda)$.
In this case, $\mathscr  D^{f-1}_{ r, s}$ consists of elements obtained by using
$s_{i}$ instead of $s_{i-1}$ in those of  usual $ \mathscr D^{f-1}_{ r, s}$ in Lemma~\ref{rcs}.
So, the required isomorphism between  $\mathfrak eC(f, \lambda)$
and $ C(f-1, \lambda)$ sends $e_1 \sigma (g_d) e_2\cdots e_f
\n_{\t\t^\lambda} +\mathscr B_{r, s}^{\rhd (f, \lambda)}$ to $\sigma(g_d)
e_2\cdots e_f \n_{\t\t^\lambda} +  {\mathscr B}_{r, s}(1)^{\rhd (f-1,
\lambda)}$. This proves (b). By  general result for rings and idempotents, we have (d).
By the adjoint associativity of Hom and tensor functors together with   (d), we have (e).

By Corollary~\ref{leftright},  the cell  module of $ {\mathscr
B}_{r+1, s+1}(1)  \mathfrak e $ with respect to $(f, \lambda)$
can be identified with $C(f, \lambda)\mathfrak e $,
where $C(f, \lambda)$ is the corresponding cell module for
${\mathscr B}_{r+1, s+1}(1)$. Note that $g_1$ is invertible. So,
 $\psi:\G(C(f, \lambda))\rightarrow C (f+1, \lambda)$
sending $h\mathfrak e \otimes  e_2\cdots e_{f+1}
\n_\lambda \mathfrak e    $ to $h e_1 e_2\cdots e_{f+1}
\n_\lambda$
 is a homomorphism as left $\mathscr B_{r+1,
s+1}$-modules. Since $e_1 e_2\cdots e_{f+1} \n_\lambda$ is a
generator of $C(f+1, \lambda)$, $\psi$  is an epimorphism.  Using
Corollary~\ref{tool2} for $\mathscr B_{r+1, s+1} e_1$ and the basis
of $C(f, \lambda)$, we have that  each element in $\G(C(f,
\lambda))$ can be written as a linear combination of elements
$$\sigma ({g_{1, i_1}} g^*_{1, j_1})\mathfrak e\otimes\sigma(g_d) e_2\cdots
e_{f+1}\n_{\t\t^\lambda}\mathfrak e ,  \ \ (\t, d)\in
\Std(\lambda)\times {\mathscr D^{f}_{r, s}}, s_{1, i_1}s^*_{1,j_1}\in
\mathscr D^{1}_{r+1, s+1},
$$ where ${\mathscr D^{f}_{r, s}}$ is obtained from usual ${\mathscr D^{f}_{r, s}}$
by using $s_{i+1}, s_{j+1}^*$ instead of $s_i$, and $s_j^*$, for $1\le i\le r-1$ and $1\le j\le s-1$, respectively.
Note that $\mathfrak e h \mathfrak e=h\mathfrak e$ for all $h\in
\mathscr B_{r+1,s+1}(1)$. By (\ref{gf})--(\ref{gf1}),  we can keep
those $s_{1, i_1}s^*_{1,j_1}\in \mathscr D^{1}_{r+1, s+1}$ such that
$s_{1, i_1}s^*_{1,j_1} d\in \mathscr D^{f+1}_{r+1, s+1}$. So,
$\dim_\kappa \G(C(f, \lambda))\le \dim_\kappa C (f+1, \lambda)$,
forcing $\psi$ to be injective. This completes the proof of (c).
 \end{proof}

\begin{Lemma}\label{subiso}  Suppose  $r\ge 2$ (resp. $s\ge 2$). Let  $\tilde {\mathscr B}_{r-1, s}$
(resp.  $\tilde {\mathscr
B}_{r, s-1}$)  be the subalgebra of $ \mathscr B_{r, s}$
generated by $g_1^{-1} e_1g_1$, $g_i$, $2\le i\le
r-1$ and  $g^*_j$, $1\le j\le s-1$ (resp.  $(g_1^*)^{-1} e_1 g_1^*$, $g_i$ and $g^*_j$ with $1\le i\le r-1$
and $2\le j\le s-1$).  Then  $\tilde {\mathscr B}_{r-1, s}\cong \mathscr B_{r-1, s}$
and $\tilde {\mathscr B}_{r, s-1}\cong \mathscr B_{r, s-1}$.
\end{Lemma}

\begin{proof}It is easy to check the required isomorphism from   $\tilde {\mathscr B}_{r-1, s}$
(resp.   $\tilde {\mathscr B}_{r, s-1}$) to $\mathscr B_{r-1, s}$ (resp.  $\mathscr B_{r, s-1}$)  sends $g_1^{-1} e_1g_1$,
(resp. $(g_1^*)^{-1} e_1 g_1^*$) and $ g_i, g^*_j$
to $e_1$ and  $g_{i-1}, g^*_j$ (resp. $g_i, g^*_{j-1}$). \end{proof}

 If   $r\geq2$,  $\mathscr B_{r-1,
s}$ is a  subalgebra of $\mathscr B_{r, s}$ with $r-1>0$. So,  we consider restriction
and induction  functors as follows:
$$\begin{aligned}
{\Res}_{r, s}^L: &\quad \mathscr{B}_{r,s}\text{-mod} \rightarrow {\mathscr{B}}_{r-1,s}\text{-mod},\\
 {\Ind}_{r-1, s}^L:  & \quad {\mathscr{B}}_{r-1,s}\text{-mod} \rightarrow \mathscr{B}_{r,s}\text{-mod}.\\
\end{aligned}$$
Similarly, we assume that $s\ge 2$, we  consider induction and restriction
functors  as follows.
$$\begin{aligned}
{\Res}_{r, s}^R: & \quad \mathscr{B}_{r,s}\text{-mod} \rightarrow {\mathscr{B}}_{r,s-1}\text{-mod},\\
{\Ind}_{r, s-1}^R: &\quad  {\mathscr{B}}_{r,s-1}\text{-mod}  \rightarrow \mathscr{B}_{r,s}\text{-mod}.\\
\end{aligned}$$

For the simplification of notations, we will use  $\Res^L$ instead of $\Res_{r, s}^L$, etc.

\begin{Lemma}\label{eequ} Let $\mathscr B_{r,s}$ be defined over $\kappa$.
 \begin{enumerate}  \item If $r\ge 2$, then $  {\mathscr
B}_{ r, s} e_1 g_1  =\tilde{ \mathscr B}_{r-1, s}e_1 g_1$. \item If
$s\ge 2$, then $\mathscr B_{r, s} e_1 g_1^*  = \tilde {\mathscr
B}_{r, s-1} e_1 g_1^*$.
\end{enumerate}
\end{Lemma}

\begin{proof}  For (a), it suffices to prove  $ {\mathscr
B}_{ r, s}e_1   \subseteq \tilde { \mathscr B}_{r-1, s}e_1$. By
Corollary~\ref{tool2}, each element in $ \mathscr B_{r, s} e_1$ can
be written as a linear combination of elements in $\sigma(g_d)
e_1 {\mathscr B}_{r, s}(1)$ where $d\in \mathscr D^1_{r, s}$.
Note that $d=s_{1, i} s^*_{1, j}$ for $i,
j\ge 1$. So, we need to verify $g_{1} e_1 \in \tilde {\mathscr
B}_{ r-1, s} e_1$. This is the case since $\rho g_1
e_1=g_1e_1g_1e_1=g_1^{-1} e_1g_1e_1+ \rho(q-q^{-1}) e_1$ and $g_1^{-1} e_1g_1\in\tilde{ \mathscr B}_{r-1, s}$.
Finally,  (b) can be proved similarly.
\end{proof}

\begin{Prop}\label{bim} Let $\mathscr B_{r, s}$ be over $\kappa$.
\begin{enumerate}  \item If $r\ge 2$, then ${\mathscr
B}_{ r, s} f_{21}\cong \tilde{ \mathscr B}_{r-1, s} $ as $(\tilde
{\mathscr B}_{r-1, s}, {\mathscr B}_{r, s}(1) f_{21})$-bimodules.
\item If $s\ge 2$, then $\mathscr B_{r, s} \tilde e_{12} \cong\tilde
{\mathscr B}_{r, s-1} $ as $(\tilde {\mathscr B}_{r-1, s},
{\mathscr B}_{r, s}(1) \tilde e_{12})$-bimodules.
\end{enumerate}
\end{Prop}
\begin{proof}
We compute the dimension  $ \mathscr B_{r, s} e_1$ via the cellular basis of $ {\mathscr B}_{r,
s}(1) $, which is obtained from that of $\mathscr B_{r-1, s-1}$ by
using $g_i, g_j^*$ and $e_2$ instead of $g_{i-1}, g^*_{j-1}$ and
$e_1$, respectively.

By Corollary~\ref{tool2}, each element in $ \mathscr B_{r, s} e_1$ can be
written as a linear combination of $\sigma (g_{d_1} g_d ) e_1 e_2\cdots e_{f}
\n_{\s\t} g_{d_{2}} $ where $\sigma (g_{d_1}) e_2\cdots e_{f}
\n_{\s\t} g_{d_{2}}$ ranges over all cellular basis elements of
$ {\mathscr B}_{r, s}(1)$ and $d\in \mathscr D_{r,s}^1$.
So, \begin{equation}\label {gde}g_d=g_{1, i_1} g^*_{1, j_1}, \text{ for some $i_1, j_1\ge 1$.}\end{equation}  We remark that
the previous $e_2\cdots e_f$ is $1$ if $f=1$.

We claim that each $\sigma(g_{d_1} g_d) e^f$ can be written as a
linear combination of $\sigma (g_a) e^f $, $a\in \mathscr D^f_{r,
s}$. In fact, we  prove the similar result for $e^f g_{d_1} g_{d}$ and
use anti-involution to get our claim. We prove it by induction on
$f$ as follows.

If $f=1$, there is nothing to be proved.  In this case, $d_1=1$. If
$f=2$, we write $d_1=s_{2, i_2} s^*_{2, j_2}$ for some $i_2, j_2\ge
2$. Since we are assuming (\ref{gde}), there is   nothing to be
proved if  $i_2>i_1$. In this case,  $d_1d \in \mathscr D_{r, s}^2$.
So, we assume
 $i_2\le i_1$. By Lemma~\ref{tool}(e) and Definition~\ref{wbmw}(g), we have
$$e_1e_2 g_{2, i_2} g^*_{2, j_2} g_{1, i_1} g_{1, j_1}^*=e_1e_2
g_{2, i_1} g_{1, i_2-1} {g_{1}^*} g^*_{2, j_2} g_{1,
j_1}^*.$$
So, our claim for $f=2$ follows from the
special case of (\ref{gf1}). Using the result for $f=2$ repeatedly
yields the result for general $f$.

 Now, we count the dimension of
$\mathscr B_{r, s}e_1$. In fact, if we use walled Brauer algebra
$B_{r,s}$ (see, e.g. in \cite{CVDM}), the classical limit of $\mathscr B_{r, s}$ instead of it,
and if we use $s_i$ and $s_j^*$ instead of  $g_i$ and $g_j^*$ in
a basis of  $\mathscr B_{r, s}e_1$, by our previous arguments, we will get a corresponding basis for  $ B_{r,s}e_1$.
So,  both $\mathscr B_{r, s}e_1$  and  $ B_{r,s}e_1$ have the same
dimension. By \cite[Prop.~2.10]{CVDM},   the dimension of $
B_{r,s}e_1$ is $(r+s-1)!$. So is $\mathscr B_{r, s}e_1$.  By Lemma~\ref{eequ}(a),   $\phi: \tilde {\mathscr
B}_{r-1, s} \rightarrow  \mathscr B_{r, s}\mathfrak e$, which  sends $h$ to $ h \mathfrak e $,   $h\in \tilde {\mathscr
B}_{r-1, s}$ is an epimorphism as left $\tilde {\mathscr B}_{r-1,
s}$-modules. Comparing the dimensions of $ \tilde {\mathscr
B}_{r-1, s}$ and $\mathscr B_{r, s}\mathfrak e
$ yields the required isomorphism as left $\tilde {\mathscr B}_{r-1,
s}$-modules.

Note that $\tilde {\mathscr B}_{r-1, s}\supset \mathscr B_{r, s}(1)$. So,   $\tilde {\mathscr B}_{r-1, s}$ is a right
$ \mathscr B_{r, s}(1)$-module. By Corollary~\ref{leftright}, it is a right $\mathscr B_{r,s}(1) \mathfrak e$-module.
More explicitly, if $h\in \tilde {\mathscr B}_{r-1, s}$ and $x\mathfrak e\in \mathscr B_{r,s}(1) \mathfrak e$ with $x\in \mathscr B_{r,s}(1) $, then the right action of
$x\mathfrak e$ on  $h$ is $h x$.
 Since $\mathfrak e x \mathfrak e=x \mathfrak e$ for any $x\in  {\mathscr B}_{r, s}(1) $, it is routine to check that
 $\phi$ is  a homomorphism  as right $\mathscr
B_{r, s}(1)\mathfrak e$-modules. This completes the proof of (a). We remark that (b) can be proved similarly.
\end{proof}

We identify  $\tilde {\mathscr B}_{r-1, s}$ (resp.   $\tilde {\mathscr B}_{r, s-1}$ ) with $\mathscr B_{r-1, s}$ (resp. $\mathscr B_{r, s-1}$) in the remaining part of this section. The following result follows from Proposition~\ref{bim}.

\begin{Prop}\label{Gind} $\Res^L \circ \G=\Ind^R$ and $\Res^R \circ \G =\Ind^L$.
\end{Prop}

Given an $(f, \lambda)\in \Lambda_{r, s}$ with $f>0$ and
$\lambda=(\lambda^{(1)}, \lambda^{(2)})$,
  $$ \mathscr
R(\lambda^{(1)})=\{p_i\mid 1\le i\le a\}, \text{ and }   \mathscr
A(\lambda^{(2)})=\{q_j\mid  1\le j\le b\}.$$
In the remaining part of this section, we always
use  $\alpha^{(i)}$ (resp. $\beta^{(j)}$ ) to denote
$(\lambda^{(1)}\setminus \{p_i\}, \lambda^{(2)})$ (resp.
$(\lambda^{(1)}, \lambda^{(2)}\cup \{q_j\})$). In other words,  $\alpha^{(i)}$ is the bipartition obtained from
$\lambda$ by removing the node $p_i$. Similarly, $\beta^{(j)}$ is the bipartition   obtained from
$\lambda$ by adding the node $q_j$.
We arrange $p_i$'s and $q_j$'s such that
\begin{equation} \label{lo} (f,\alpha^{(1)})\rhd(f,\alpha^{(2)})\rhd\cdots\rhd (f,\alpha^{(a)})
\rhd (f-1, \beta^{(1)})\rhd\cdots\rhd
(f-1,\beta^{(b)}).
\end{equation}

\begin{Lemma} \label{nk}  Suppose  $(f, \lambda)\in \Lambda_{r,s}$ with
$\lambda=(\lambda^{(1)}, \lambda^{(2)})$.  Let
  $y_{\alpha^{(k)}}= g_{r, a_k}\n_{\lambda} e^f +\mathscr{B}_{r,s}^{\rhd(f,\lambda)}\in
C(f,\lambda)$ where
$a_k=f+\sum_{j=1}^\ell \lambda^{(1)}_j$ if $p_k\in \mathscr R(\lambda^{(1)})$ with  $p_k=(\ell,
\lambda^{(1)}_\ell)$ for some $\ell$. Then, there is an epimorphism
 $C(f, \alpha^{(k)} )\twoheadrightarrow N_k/N_{k-1}$, where $N_k=\sum_{j=1}^k \mathscr{B}_{r-1,s} y_{\alpha^{(j)}}
$, $1\le k\le a$ and $N_0=0$.\end{Lemma}

\begin{proof} Since  $y_{\alpha^{(k)}}\in C(f, \lambda)$, we have $N_k\subset C(f,
\lambda)$.  By Corollary~\ref{tool2},  $\mathscr B_{r-1, s} y_{\alpha^{(k)}}$ is spanned by $
\sigma(g_d)  \mathscr B_{r-1, s}(f)  y_{\alpha^{(k)}} $, $d\in \mathscr
D_{r-1, s}^f$.

 On the other hand, we have  $h e^f \equiv
\psi_f(h) e^f \pmod {\mathscr B_{r,s}^{f+1}}$ for any $h\in
{\mathscr B}_{r-1, s}(f)$, where $ \psi_f: {\mathscr B}_{r-1,
s}(f)\rightarrow \H_{r-1-f}\otimes \H_{ s-f}$ is the epimorphism with
kernel $\langle e_{f+1}\rangle$, the two-sided ideal of ${\mathscr
B}_{r-1, s}(f)$ generated by $e_{f+1}$. Using the branching rule for
the cell module $C(\lambda^{(1)})$ for the Hecke algebra $ \H_{r-f}$
(see, e.g. \cite{Ma}), we have that $N_k/N_{k-1}$ is spanned by all
$ \sigma (g_{d}) \sigma ({g^*_{d(\t)}}) \sigma (g_{d(\v)})
y_{\alpha^{(k)}}+N_{k-1}$
 where $\v\in \Std(\lambda^{(1)}\setminus \{p_k\})$,
 $\t\in
\Std(\lambda^{(2)})$ and $d\in \mathscr D_{r-1,
s}^f=\{d\in\mathscr{D}_{r,s}^f|(r) d = r \}$. Note that $
y_{\alpha^{(k)}}=\n_{\alpha^{(k)}} e^f h$, where
$$h=g_{r, a_k} \sum_{i=b_k+1}^{a_k} (-q)^{a_k-i} g_{a_k, i}\in \mathscr
B_{r, s},$$ where $b_k=f+\sum_{i=1}^{\ell-1}\lambda^{(1)}_i$.
So, the required epimorphism sends $\sigma (g_{d}) \sigma
({g^*_{d(\t)}}) \sigma(g_{d(\v)})\n_{\alpha^{(k)}} e^f$ to $\sigma
(g_{d}) \sigma ({g^*_{d(\t)}}) \sigma(g_{d(\v)})\n_{\alpha^{(k)}}
e^f h$.  \end{proof}

We need some combinatorial preparations before we prove the result on
the branching rule for $\mathscr B_{r, s}$. This is motivated by Enyang's work on Birman-Murakami-Wenzl algebras in \cite{Enyang1}

Recall that a composition $\mu$ of $n$ is a sequence of non-negative
integers $(\mu_1, \mu_2, \ldots)$ with $\sum_i \mu_i=n$. Given  a
partition $\lambda$ and a composition $\mu$ of $n$,  a
$\lambda$-tableau $\ssS$ of {\it content} (or {\it type}) $\mu$ is
the tableau obtained from $Y(\lambda)$ by inserting each box with
numbers $i, 1\le i\le n$, such that the number $i$ occurring in
$\ssS$ is $\mu_i$. If the entries in $\ssS$ are weakly increasing in
each row  and strictly increasing in each column, $\ssS$ is called a
semi-standard $\lambda$-tableau of content $\mu$. Let
$\bfT^{ss}(\la,\mu)$ be the set of all semi-standard
$\lambda$-tableaux of content $\mu$. If $\bfT^{ss}(\lambda,
\mu)\neq\emptyset$, then $\lambda\unrhd\mu$.

Let $\s$ be a $\lambda$-tableau and let $\mu$ be a composition. Then
$\mu(\s)$ is the $\lambda$-tableau of type $\mu$ which is obtained
from $\s$ by replacing each entry $i$  in $\s$ by $j$ if $i$ appears
in row $j$ of $\t^\mu$.

Suppose $\s\in \Std(\lambda)$ and $\lambda\in \Lambda^+(n)$. Let
$\s\!\downarrow_i$ be obtained from $\s$ by removing all entries
which are strictly bigger than $i$. Then $\s\!\downarrow_i$ is a
standard $\mu$-tableau for some partition $\mu\in \Lambda^+(i)$. In
this case, we use $\s_i$ instead of $\mu$.

The following result, which has already been used in the proof of
\cite[Coro.~5.4]{Enyang1}, can be verified, easily.

\begin{Lemma}\label{big} Assume  $\lambda, \mu\in \Lambda^+(n)$ with
$\mu_i=1$ and $i=l(\mu)$. Suppose $\ssS\in \bfT^{ss}(\lambda, \mu)$
and $\s\in \Std(\lambda)$ such that $\mu(\s)=\ssS$. Then
$\s_{n-1}\unrhd \nu$ where $\nu$ is obtained from $\mu$ by removing
the removable node with maximal row index. Further, $\s_{n-1}=\nu$
if and only if $\lambda $ is obtained from $\nu$ by adding an
addable node.
\end{Lemma}


\begin{Lemma}\label{zbeta} For each $k$, $1\le k\le b$, define
$$z_{\beta^{(k)}} = \sum_{j=d_k}^{c_k} (-q)^{j-c_k} g^*_{j, c_k}  (g^*_{f, c_k})^{-1} g_{r, f}\n_\lambda
e^f+ \mathscr{B}_{r,s}^{\rhd(f,\lambda)}\in C(f,\lambda),$$ where   $c_k=f+\sum_{j=1}^\ell \lambda_j^{(2)}$,  $d_k=f+\sum_{j=1}^{\ell-1}
 \lambda^{(2)}_j+1$
if
 $\beta^{(k)}=(\lambda^{(1)},  \lambda^{(2)}\cup \{q_k\})$ with  $q_k=(\ell, \lambda_\ell^{(2)}+1)\in \mathscr A(\lambda^{(2)})$.
 Then  $z_{\beta^{(k)}}\in
\mathscr B_{r-1, s} z_{\beta^{(b)}}$ for any $k, 1\le k\le b$.\end{Lemma}
\begin{proof} By definition,  \begin{equation}\label{zb} z_{\beta^{(b)}}  =  g_{r,f} (g^*_{f,
 s})^{-1}  \n_{\lambda} e^{f} +\mathscr B_{r, s}^{\rhd(f, \lambda)}.\end{equation}
 It is routine to check $z_{\beta^{(k)}} =\sum_{j=d_k}^{c_k} (-q)^{j-c_k} g_{j, c_k}^*
 g_{c_k, s}^*z_{\beta^{(b)}}\in
\mathscr B_{r-1, s} z_{\beta^{(b)}}$ for any $k, 1\le k\le b$.\end{proof}

\begin{Lemma}\label{eqi0}
Suppose $(f-1, \mu)\in \Lambda_{r-1, s}$ with
$\mu=(\lambda^{(1)}, \mu^{(2)})\unrhd \beta^{(b)}$
and $\mu\neq \beta^{(i)}$, $1\le i\le b$.  Write $\beta^{(b)}=(\lambda^{(1)}, \nu)$.
If $\t\in \Std(\lambda^{(1)})$ and  $\s\in \nu^{-1} (\ssS)$ with $\ssS\in \bfT^{ss}(\mu^{(2)}, \nu)$, then
$\n_{\mu^{(2)}} g^*_{d(\s) }
 (g^*_{f, s})^{-1} g_{r, f} \n_{\t\t^{\lambda^{(1)}}} e^f \in  \mathscr B_{r, s}^{\rhd (f, \lambda)}$.
\end{Lemma}

\begin{proof}
For any $\s\in \nu^{-1}(\ssS)$ let $\u=\s\!\!\downarrow_{s-1}$. By
Lemma~\ref{big}, $\s_{s-1}\rhd \lambda^{(2)}$. For the
simplification of notation, we denote $\s_{s-1}$ by $\tau$. We write
$d(\s)=s_{b_k, s} d(\u)$ where  $b_k=f-1+\sum_{i=1}^k \mu^{(2)}_i$
if we assume that  $s$ is in the $k$th row of $\s$. So,
$$\n_{\mu^{(2)}} g^*_{d(\s)} (g^*_{f, s})^{-1}=\sum_{j=b_{k-1}+1}^{b_k}(-q)^{j-b_k} g^*_{j, b_k}   (g^*_{f, b_k})^{-1} \n_\tau g^*_{d(\bar\u)}, $$
where $\bar \u$ is obtained by using $i$ instead of  $i-1$ in $\u$ for all possible $i$'s. Now, the result follows from $\tau \rhd \lambda^{(2)}$.
\end{proof}

\begin{Lemma}\label{farc} For any  $b\in \mathscr{B}_{r-1,s}  e^{f-1}  \cap\mathscr {B}_{r-1,s}^f$, \begin{equation}\label{keyy}b \n_{\beta^{(b)}}(g_{f, s}^*)^{-1} g_{r, f}  e_f+\mathscr B_{r, s}^{\rhd{f, \lambda}}\in N_a.\end{equation}
\end{Lemma}

 \begin{proof} By Corollary~\ref{tool2}, any $b\in \mathscr{B}_{r-1,s}  e^{f-1} $ can be written as a linear combination of elements
$\sigma (g_d) e^{f-1} {\mathscr B}_{r-1, s}(f-1)$ with $d\in
\mathscr D_{r-1, s}^{f-1}$.  Then we use the cellular basis of
${\mathscr B}_{r-1, s}(f-1)$ to write any $b\in \mathscr{B}_{r-1,s}
e^{f-1} \cap\mathscr {B}_{r-1,s}^f$ as a linear combination of
elements
 \begin{equation} \label{key-1}\sigma (g_d) e^{f-1}  \sigma (g_{f, i} g^*_{f, j} ) h e_f g_{f, i_1} g^*_{f, j_1}=\sigma (g_d)\sigma (g_{f, i} g^*_{f, j} ) he^f g_{f, i_1} g^*_{f, j_1}
\end{equation}
with $f\le i, i_1\le r-1$, $f\le  j, j_1\le s$ and $h\in \mathscr B_{r-1, s}(f)$.
 We denote  $b$ by one of  elements in \eqref{key-1}. Note that
 \begin{equation}\label{key-2}  \n_{\beta^{(b)}}
 g_{r,f} (g^*_{f,
 s})^{-1}= g_{r,f} (g^*_{f,
 s})^{-1}  \n_{\lambda}.\end{equation}
where $  \n_{\lambda}\in \mathscr B_{r, s}(f)$ and $\n_{\beta^{(b)}}\in \mathscr B_{r-1, s}(f-1)$ (see \eqref{lo}). In order to prove \eqref{keyy}, by \eqref{key-1},
it suffices to prove
\begin{equation}\label{keyyy} e^f g_{f, i_1} g^*_{f, j_1} (g^*_{f, s})^{-1} g_{r,f}  e_f \n_{\lambda}   +\mathscr B_{r, s}^{\rhd{f, \lambda}}  \in N_a.\end{equation}
 We have
 $$
 e^f g_{f, i_1} g^*_{f, j_1} (g^*_{f, s})^{-1}
g_{r, f} e_f\n_\lambda  =\begin{cases} (g_{f+1, s}^*)^{-1} g_{r, f+1} e^{f+1} g^*_{f+1,
j_1+1}\n_\lambda , & \text{ if $j_1<s$},\\
\rho g_{r,
f+1} g_{f+1, i_1+1} e^f \n_\lambda,  & \text{ if $j_1=s$}.\\
\end{cases}
 $$
  Obviously, if $j_1<s$, then  $e^f g_{f, i_1} g^*_{f, j_1} (g^*_{f, s})^{-1} g_{r,f}  e_f \n_{\lambda}   +\mathscr B_{r, s}^{\rhd{f, \lambda}} =0\in N_a$.

 If $j_1=s$, we have $ g_{r, f+1} g_{f+1, i_1+1} n_{\lambda^{(1)}}\in C(\lambda^{(1)})$, where   $C(\lambda^{(1)})$
 is the cell module of $\mathscr H_{r-f}(f)$ with respect to $\lambda^{(1)}$ in section~3. So, \eqref{keyyy} follows from
   branching rule for cell module $C(\lambda^{(1)})$  of Hecke algebra $\mathscr H_{r-f}$ in \cite{Ma}. \end{proof}

Suppose  $\s$ is a standard tableau with entries in $f+1, f+2,
\cdots, r$.  In the following, let  $\tilde\s$ be the standard
tableau  obtained from $\s$ by using $i-1$ instead of $i$ in $\s$,
for all $i$,  $f+1\le i\le r$.

\begin{Lemma} \label{ab} For each $k, 1\le k\le b$, let  $M_k=N_a+\sum_{j=1}^k \mathscr B_{r-1, s} z_{\beta^{(j)}}$, where
$z_{\beta^{(j)}}$ is defined in Lemma~\ref{zbeta}. Then
 $M_b/N_a$ is spanned by $\{   \sigma (g_v) \sigma(g^*_{d(t)})\sigma(g_{d(\tilde
\s)}) z_{\beta^{(k)}} +N_a| (\tilde s, \t) \in \Std(\beta^{(k)}),
1\le k\leq b, v\in \mathscr{D}_{r-1,s}^{f-1} \}$.
\end{Lemma}

\begin{proof} For any $k, 1\le k\le b$, we have
$$\begin{aligned} z_{\beta^{(b)}} & =  g_{r,f} (g^*_{f,
 s})^{-1}  \n_{\lambda} e^{f} +\mathscr B_{r, s}^{\rhd(f, \lambda)}=\n_{\beta^{(b)}}
 g_{r,f} (g^*_{f,
 s})^{-1}e^f +\mathscr B_{r, s}^{\rhd(f, \lambda)},\\
 z_{\beta^{(k)}} & =\sum_{j=d_k}^{c_k} (-q)^{j-c_k} g_{j, c_k}^*
 g_{c_k, s}^*z_{\beta^{(b)}}, \\
 \end{aligned}
 $$
where $c_k$ and $d_k$ are defined in Lemma~\ref{zbeta}. So, $M_b$ is
generated by $z_{\beta^{(b)}}$ and $N_a$. Note that
$$ \mathscr B_{r-1, s}  z_{\beta^{(b)}} = \mathscr B_{r-1, s}  e^{f-1} \n_{\beta^{(b)}}
 g_{r,f} (g^*_{f,
 s})^{-1}e_f +\mathscr B_{r, s}^{\rhd(f, \lambda)}.$$
Applying
Corollary~\ref{tool2} to $\mathscr B_{r-1, s}e^{f-1}$, we have that
$ \mathscr B_{r-1, s}  z_{\beta^{(b)}} $ is spanned by $$
\sigma(g_d) e^{f-1} {\mathscr B}_{r-1, s}(f-1) \n_{\beta^{(b)}}
 g_{r,f} (g^*_{f,
 s})^{-1}e_f +\mathscr B_{r, s}^{\rhd(f, \lambda)},$$
 where $d$ ranges over all elements in $\mathscr D^{f-1}_{r-1, s}$.
For $x\in   \mathscr {B}_{r-1,s}(f-1)$, we can assume that  $x$ is in either    $ \H_{r-f}\otimes \H_{s-f+1} $
or $ \mathscr {B}_{r-1,s}^1(f-1)$, where $\mathscr {B}_{r-1,s}^1(f-1)$ is the two-sided ideal of
$\mathscr {B}_{r-1,s}(f-1)$ generated by $e_f$.

If $x\in  \mathscr {B}_{r-1,s}^1(f-1)$, then $\sigma(g_d) e^{f-1} x\in \mathscr B_{r-1, s } e^{f-1}\cap \mathscr B_{r-1, s}^f$.
By Lemma~\ref{farc},  $$
\sigma(g_d) e^{f-1} x \n_{\beta^{(b)}}
 g_{r,f} (g^*_{f,
 s})^{-1}e_f +\mathscr B_{r, s}^{\rhd(f, \lambda)}\in N_a.$$

Write $\beta^{(i)}=(\lambda^{(1)}, \gamma^{(i)})$. If $x\in \H_{r-f}\otimes\H_{ s-f+1}$, then   $ e^{f-1}
x \n_{\beta^{(b)}}
 g_{r,f} (g^*_{f,
 s})^{-1}e_f +\mathscr B_{r, s}^{\rhd (f, \lambda)}$ can be written
 as a linear combination of elements $e^{f-1} \n_{\tilde \s\tilde \v}  \n_{\t
 \u} g_{r, f} (g^*_{f, s})^{-1}e_f+\mathscr B_{r, s}^{\rhd (f, \lambda)}$ where $\tilde \s\in \Std(\lambda^{(1)})$, $\tilde \v=\tilde \t^{\lambda^{(1)}}$,  $\t\in \Std(\gamma^{(b)})$,  $\u\in (\gamma^{(b)})^{-1} (\ssS)$ with
  $\ssS\in \bfT^{ss}(\mu,\gamma^{(b)} ) $.
By Lemma~\ref{eqi0}, we can assume $ \mu\in \{\gamma^{(i)}\mid 1\le
i\le b\}$ in \eqref{lo}.  Further, if  $\mu\in \{\gamma^{(i)}\mid 1\le i\le b\}$, then
there is a unique $\ssS\in \bfT^{ss}(\gamma^{(i)}, \gamma^{(b)})$ such that
the type of $T$ is $\lambda^{(2)}$ where $T$ is obtained from $S$ by removing the node containing the  unique largest entry.
 Further, there is a unique $\u\in (\gamma^{(b)})^{-1} (S)$
such that $d(\u)=s^*_{c_i, s}$, where $c_i$ is defined in Lemma~\ref{zbeta}.
So, $$\begin{aligned} e^{f-1} & \n_{\tilde \s\tilde \v}  \n_{\t
 \u} g_{r, f} (g^*_{f, s})^{-1}e_f \equiv \n_{\tilde \s\tilde \v}  \n_{\t \t^{\gamma^{(i)}}} g_{c_i, s}^*
  g_{r, f} (g^*_{f, s})^{-1}e^f
\\ & \equiv \sigma(g_{d(\tilde\s)})\sigma(g^*_{d(\t)})  \n_{\lambda^{(1)}}\n_{\gamma^{(i)}} (g_{f, c_i}^*)^{-1} g_{r, f} e^f\\
& \equiv \sigma(g_{d(\tilde\s)})\sigma(g^*_{d(\t)}) \sum_{j=d_i}^{c_i} (-q)^{j-c_i} g^*_{j, c_i} (g_{f, c_i}^*)^{-1} g_{r, f} \n_{\lambda} e^f + \mathscr B_{r, s}^{\rhd(f, \lambda)} \\
& = \sigma(g_{d(\tilde\s)})\sigma(g^*_{d(\t)})z_{\beta^{(i)}}
\end{aligned} $$

 Therefore,  $M_b/N_a$ is
spanned by the elements, as required.
\end{proof}

The following result can be proved similarly.
\begin{Lemma} Let $M_0=N_a$.  Then $M_k/M_{k-1}$
is  spanned by $\{\sigma (g_v) \sigma(g^*_{d(t)})\sigma(g_{d(\tilde
\s)}) z_{\beta^{(k)}} +M_{k-1}| (\tilde \s, \t) \in
\Std(\beta^{(k)}),  v\in \mathscr{D}_{r-1,s}^{f-1} \}$, for all $k, 1\le k\le b$.
\end{Lemma}

\begin{Theorem}\label{branching} Suppose $(f, \lambda)\in \Lambda_{r, s}$ with $\lambda=(\lambda^{(1)},
\lambda^{(2)})$. Let $N_i$, $1\le i\le a$ be defined in Lemma~\ref{nk}.
For $1\le j\le b$, Let $N_{a+j}=M_j$, where $M_j$ is defined in Lemma~\ref{ab}. Then
$$0\subset N_1\subset N_2\subset\cdots \subset N_a\subset N_{a+1}\subset \cdots \subset N_{a+b}=C(f, \lambda)$$ is a
filtration of $\mathscr B_{r-1, s}$-modules such that
$$N_{i}/N_{i-1}\cong \begin{cases} C(f, \alpha^{(i)}), & \text{if $1\le i\le a$,}\\
                                     C(f-1, \beta^{(i-a)}),  & \text{if $a+1\le i\le a+b$.}\end{cases} $$\end{Theorem}

\begin{proof} We prove our result under the assumption $f>0$. Otherwise, since
$C(0, \lambda)\cong  S_\alpha\otimes  S_\beta$ where $\alpha $ and
$\beta$ are conjugates of  $\lambda^{(1)}$ and $\lambda^{(2)}$,
respectively, the result follows from the corresponding result for
Hecke algebra. See, e.g \cite{Ma}.

We have constructed a filtration of $M_b$ such that there is an
epimorphism from $C(f, \alpha^{(k)})$  (resp. $ C(f-1, \beta^{(k)}) $)
to $N_k/N_{k-1}$ (resp. $M_k/M_{k-1}$). We claim $M_b=C(f,
\lambda)$.

In fact, by definition of $M_b$, we have $ M_b\subseteq C(f,
\lambda)$. Note that any element in $C(f, \lambda)$ can be
expressed as a linear combination of elements $\sigma(g_d)\sigma(g_{d(\s)}) \sigma(g^*_{d(\t)}) \n_{\lambda} e^f+\mathscr B_{r, s}^{\rhd(f, \lambda)}$,   where $d\in
\mathscr D_{r,s}^f$ and $(\s, \t)\in \Std(\lambda)$. Further,
$d=s_{f, i_f} s^*_{f, j_f}\cdots s_{1, i_1} s^*_{1, j_1}$ with $i_k,
j_k\ge k$ and $i_f>i_{f-1}>\cdots>i_1$. If $i_f=r$,  then $$\sigma(g_d) \sigma(g_{d(\s)}) \sigma(g^*_{d(\t)})
\n_\lambda e^f\equiv
\sigma(g_{d_1}) \sigma(g_{d(\tilde \s)}) \sigma(g^*_{d(\t)}) g_{r, f} \n_\lambda e^f+\mathscr B_{r, s}^{\rhd(f, \lambda)}
\in \mathscr B_{r-1, s} z_{\beta^{(b)}} \subset M_b,$$
where $d_1=s^*_{f, j_f}s_{f-1, i_{f-1}}s^*_{f-1, j_{f-1}}\cdots s_{1, i_{1}}s^*_{1, j_{1}}$.
We remark that the inclusion follows from \eqref{zb}. If $i_f<r$, then   $$\sigma(g_d)\sigma(g_{d(\s)})\sigma(g^*_{d(\t)})
\n_\lambda e^f +\mathscr B_{r, s}^{\rhd(f, \lambda)}\in N_a\subset
M_b.$$ So,  $ M_b\supseteq C(f, \lambda)$ and hence $M_b=C(f,
\lambda)$.
So, $$
\dim_\kappa C(f, \lambda)=\sum_{i=1}^{a+b} N_i/N_{i-1}
\le \sum_{i=1}^a \dim_\kappa C(f, \alpha^{(i)})+
\sum_{j=1}^b \dim_\kappa C(f-1, \beta^{(j)})=\dim_\kappa C(f, \lambda),$$
where the last equality follows from  branching rule for cell modules for walled Brauer algebras in
\cite[Theorem~3.3]{CVDM}. So, $\dim_\kappa C(f, \alpha^{(i)})=\dim_\kappa N_i/N_{i-1}$ and $ \dim_\kappa C(f-1,
\beta^{(j)})=\dim_\kappa M_j/M_{j-1}$, forcing  $C(f, \alpha^{(i)})\cong N_i/N_{i-1}$ and $ C(f-1,
\beta^{(j)})\cong M_j/M_{j-1}$, for all possible $i$ and $j$. \end{proof}

\section{The irreducible $\mathscr{B}_{r,s}$-modules}\label{classification}
 In this section, we   classify the irreducible $\mathscr B_{r, s}$-modules
over an arbitrary field $\kappa$. First, we   briefly recall the representation theory
of cellular algebras~\cite{GL}.  At moment, we keep the notations in
Definition~\ref{GL}. So, $A$ is a cellular algebra over a
commutative ring $R$ containing $1$ with a cellular basis $\set{C_{\s\t}|\s,\t\in
T(\lambda), \lambda\in\Lambda}$. Unlike what we have done in section~4, we consider
the right $A$-module in this section. As we mentioned before, each left $A$-module can be considered as
as a right $A$-module. The motivation for using right $\mathscr B_{r, s}$-module is that bases of right cell modules of
 $\mathscr B_{r, s}$ can be used to classify singular vectors in the mixed tensor product of natural module and its dual over
 $U_q(\mathfrak{gl}_n)$. Details will be given in \cite{Rsong}.

Recall that each cell module $C(\lambda)$ of $A$ is the free $R$-module with basis
$\{C_\bfs\mid \bfs\in T(\lambda)\}$. In \cite{GL}, Graham and Lehrer have proved that every irreducible
$A$--module arises in a unique way as the simple
head of some cell module.  More explicitly, each
$C(\lambda)$ comes equipped with the invariant form
$\phi_{\lambda}$ which is determined by the equation
$$C_{\bfs\bft}
           C_{\bft'\bfs}
 \equiv\phi_{\lambda}\big(C_{\bft},
               C_{\bft'}\big)\cdot
        C_{\bfs\bfs}\pmod{ A^{\rhd \lambda}}.$$
 Consequently,
$$\Rad C(\lambda)
   =\set{x\in C(\lambda)|\phi_{\lambda}(x,y)=0\text{ for all }
                          y\in C(\lambda)}$$
is an $A$--submodule of $C(\lambda)$ and $D^\lambda=C(\lambda)/\Rad
C(\lambda)$ is either zero or absolutely irreducible. Graham and
Lehrer~\cite{GL} have proved the following result in \cite{GL}.

\begin{Theorem} \cite{GL}\label{cell-tool} Let $(A, \Lambda)$ be a cellular algebra over a field $\kappa$.
\begin{enumerate} \item  The set  $\{D^\lambda\mid D^\lambda\neq 0\}$  consists of a
complete set of pairwise non-isomorphic irreducible $A$-modules.
\item Let $G_\lambda$ be the Gram matrix with respect to the
invariant form $\phi_\lambda$ on $C(\lambda)$. Then $A$ is split
semisimple if and only if $\prod_{\lambda\in \Lambda}\det
G_\lambda\neq 0$ in $\kappa$. \end{enumerate}\end{Theorem}

We remark that we will use Theorem~\ref{cell-tool} frequently in
sections~5-6.  Via Theorem~\ref{celb}, we have the notion of cell
modules for $\mathscr B_{r, s}$. We use Theorem~\ref{cell-tool}
 to classify the irreducible $\mathscr B_{r, s}$-module over $\kappa$. Let    $\phi_{f,
\lambda}$ be  the corresponding invariant form on $C(f, \lambda)$,
$(f, \lambda)\in \Lambda_{r, s}$.

By abuse of notations, we use  $ \mathscr H_{r-f}$ (resp. $\mathscr H_{s-f}$)
 to denote the subalgebra of ${\mathscr B}_{r,s}(f)$ in Lemma~\ref{quo1}. Then  $\{\n_{\s\t}|\s,
\t\in\Std(\lambda),\lambda\in\Lambda_{r,s}^f \}$ is a cellular basis of   $\mathscr H_{r-f}\otimes \mathscr H_{s-f}$.
 Let $\phi_{\lambda}$ be the invariant form on the cell module  $C(\lambda)$ of
$\mathscr H_{r-f}\otimes \mathscr H_{s-f}$ with respect to
$\lambda\in \Lambda_{r,s}^f$. In the following, we denote $\mathscr
H_{r-f}\otimes \mathscr H_{s-f}$  by $\mathbf H(f)$ and  use $\n_\t$
to denote $\n_{\t^\lambda\t}+\mathbf H(f)^{\rhd \lambda}$

\begin{Lemma}\label{bilequal} Let $\mathscr B_{r,s}$ be defined over $\kappa$.  Suppose  $(f,\lambda)\in\Lambda_{r,s} $.
\begin{enumerate} \item  If either  $r\neq s $ or $r=s$ and $f<r$, then $\phi_{f,\lambda}\neq 0\Leftrightarrow\phi_{\lambda}\neq 0$,
\item If $r=s=f$,  then $\phi_{f,0}=0 \Leftrightarrow \delta=0$.\end{enumerate}
\end{Lemma}

\begin{proof} If  $\phi_{\lambda}\neq 0$, then  $\phi_{\lambda} (\n_\s, \n_\t)\neq 0$
 for some  $\s, \t\in \Std(\lambda)$.  We have $\phi_{f,\lambda}\neq 0$ since
 $$\n_{\ts^\lambda \s}e^f g^*_{f+1,1} e^f \n_{\t \t^\lambda}  \equiv\phi_{\lambda}(\n_{\s}, \n_{\t
})\rho^{f}e^f \n_{\t^\lambda\t^\lambda} \pmod{
\mathscr{B}_{r,s}^{\rhd(f,\lambda)}}. $$

If $\phi_{f,\lambda}\neq 0$, then  $\phi_{f, \lambda}(e^f\n_{\s}
g_{d}, e^f \n_{\t} g_{e} )\neq 0$ for some $(\s, d), (\t, e)\in I(f,
\lambda)$. We have $\phi_{\lambda}\neq 0$. Otherwise,
$\n_{\t^\lambda \s} h \n_{\t\t^\lambda}\equiv 0 \pmod { \mathbf
H(f)^{\rhd \lambda}}$, for all $h\in \mathbf H(f)$. Since
$$e^f\n_{\t^\lambda \s}g_d \sigma (g_e)e^f \n_{\t\t^\lambda}\equiv
\n_{\t^\lambda \s} h \n_{\t\t^\lambda} e^f \pmod {\mathscr B_{r,
s}^{f+1}}, $$ we have   $\phi_{f, \lambda}(e^f\n_{\s} g_{d}, e^f
\n_{\t} g_{e} )= 0$ , a contradiction.
 This completes the proof of (a).

 Suppose  $r=s=f$ and   $\delta=0$.  Using
Lemma~\ref{efe}(a) repeatedly yields the following inclusion:
$e^fg_dg_e^*e^f\subseteq e^{f-1} e_f  {\mathscr B}_{f,f}(f-1)e_f$.
Therefore, we need to verify  $e_f {\mathscr B}_{f,f}(f-1)e_f=0$,
which is equivalent to the equality $e_1 \mathscr B_{1,1} e_1=0$. In
fact, this follows since $\mathscr B_{1,1}=\{1, e_1\}$ and
$\delta=0$.  Conversely, the result follows from the equalities $e^f
e^f=\delta^f e^f=0$.
\end{proof}

Let $e$ be the least positive integer such that $1+q^2+\cdots +q^{2(e-1)}=0$ in $\kappa$. If there
is no such an $e$, we set $e=\infty$.

   Suppose
$\lambda=(\lambda_1, \lambda_2, \cdots)$ is a partition. Recall that
$\lambda$ is $e$-restricted if $\lambda_i-\lambda_{i+1}<e$ for all
possible $i$. If $\lambda=(\lambda^{(1)}, \lambda^{(2)})$, then
$\lambda$ is said to be  $e$-restricted if both $\lambda^{(1)}$ and
$\lambda^{(2)}$ are $e$-restricted. The following result follows
from Lemma~\ref{bilequal}, immediately.

\begin{Theorem}\label{main1}
Let $\mathscr{B}_{r,s}$ be the quantized walled Brauer algebra over
the field $\kappa$.
\begin{enumerate}\item If either $\delta\neq 0$ or $\delta=0$ and $r\neq s$,
then the non-isomorphic irreducible $\mathscr{B}_{r,s}$--modules are
indexed by $\{(f, \lambda)\mid  0\le f\le \min\{r, s\}, \lambda \text{ being $e$-restricted}\}$.
\item If  $\delta=0$ and $r=s$,  then the non-isomorphic irreducible
$\mathscr{B}_{r,s}$--modules are indexed by  $\{(f, \lambda)\mid  0\le f<r, \lambda \text{ being $e$-restricted}\}$.
\end{enumerate}
\end{Theorem}

\begin{rem}  Enyang~\cite{Enyang2} classified the irreducible $\mathscr B_{r, s}$-modules by
using the conditions $D^{f, \lambda}\neq 0$. However, there is no
further information about $(f, \lambda)$ in \cite{Enyang2}.
\end{rem}
The following result follows from Theorem~\ref{main1} and \cite[3.10]{GL}.
\begin{Cor}\label{main11}  Let $\mathscr{B}_{r,s}$ be the quantized walled Brauer algebra over
the field $\kappa$. Then $\mathscr B_{r,s}$ is quasi-hereditary in
the sense of \cite{CPS} if and only if $e> \max\{r,s\}$ and either
$\delta\neq 0$,  or $\delta=0$ and $r\neq s$.
\end{Cor}

\section{A criterion on the semi-simplicity of $\mathscr B_{r, s}$}
In this section, we give a necessary and sufficient condition for
$\mathscr B_{r, s}$ being semisimple over an arbitrary field
$\kappa$. We start by recalling Kosuda and Murakami's result as
follows.

\begin{Lemma}\label{KoM}\cite[6.7]{KM} Let $\mathscr B_{r, s}$ be defined over $\mathbb C$ with
$\rho=q^n$ and $e=\infty$. Then $\mathscr B_{r, s}$ is semisimple if
$n\ge r+s$.\end{Lemma}

Let $c_{r, s}\in \mathscr B_{r, s}$ be defined in Proposition~\ref{ccen}.  We want to compute the action of
$c_{r, s}$  on each cell module of $C(f, \lambda)$ for all $(f,
\lambda)\in \Lambda_{r, s}$. We remark that
all cell modules in this sections are left cell modules. Via anti-involution $\sigma$ in Lemma~\ref{inv}, they can be considered as
right modules.

\begin{Defn} Suppose $(f, \lambda)\in \Lambda_{r, s}$ with
$\lambda=(\lambda^{(1)}, \lambda^{(2)})$. \begin{enumerate} \item If
$p$ is in the $i$th row and $j$th column of the Young diagram
$[\lambda^{(1)}]$, we define $$c(p)=\frac{1-q^{2k}}{q-q^{-1}},$$
where $k=\res(p)=j-i$.
\item If
$p$ is in the $i$th row and $j$th column of the Young diagram
$[\lambda^{(2)}]$, we define $$c(p)=\frac{1-q^{-2k}}{q^{-1}-q}, $$
where $k=\res(p)=j-i$.

\end{enumerate}
\end{Defn}

\begin{Lemma}\label{resi}  For  $(f, \lambda)\in \Lambda_{r, s}$, let
 $C(f, \lambda)$ be the cell module of $\mathscr B_{r,s}$ with respect to the cellular basis in Theorem~\ref{celb}.
  Then $c_{r, s}$ acts on  $C(f, \lambda)$
as scalar $f\delta-\rho^{-1}\sum_{p\in [\lambda^{(1)}]}
c(p)-\rho\sum_{p\in [\lambda^{(2)}]} c(p)\in R$.
 \end{Lemma}
\begin{proof} For any $(f, \lambda)\in \Lambda_{r, s}$, by Lemma~\ref{KoM}, $\det G_{f, \lambda}\neq 0$ under
some  specialization of  $R$. This implies that $\det G_{f,
\lambda}\neq 0$ over the field $F$ of the fractions of $R$. So,
$\mathscr B_{r, s}$ over $F$ is semisimple  and $C(f, \lambda)$ is
irreducible. Since $c_{r, s}$ is central, it acts on  $C(f,
\lambda)$ as scalar over $F$. Since both $C(f, \lambda)$ and $c_{r,
s}$ are defined over $R$, $c_{r, s}$ acts on $C(f, \lambda)$ as
scalar over $R$, too. So, we need only prove
\begin{equation}\label{identity} c_{r, s} y=(f\delta-\rho^{-1}\sum_{p\in [\lambda^{(1)}]}
c(p)-\rho\sum_{p\in [\lambda^{(2)}]} c(p)) y\end{equation} where
$y=e^f \n_\lambda+\mathscr B_{r, s}^{\rhd (f, \lambda)}$. It is easy to see that   $e_{i, i } e^f= \delta e^f$ for  $1\le i\le f$ and   $\bar e_{i, j} e^f \in \mathscr B_{r, s}^{f+1}$,  if $f<i, j$.
Further,
$$\bar e_{i, j } e^f=\begin{cases}
\rho^{-1} g_{j, i}^{-1} g_{i, j+1}^{-1} e^f, & \text{if
$1\le j\le f<i$,}\\
\rho  g^*_{j, i} g^*_{i+1, j} e^f, & \text{if $1\le i\le
f<j$,}\\
\rho g_{j, i}^*
g_{i+1, j}^*  e^f, & \text{if  $1\le i<j\le f$,}\\
\rho^{-1} g_{j, i}^{-1} g_{i, j+1}^{-1} e^f, & \text{if $1\le j<i\le
f$.}\\
\end{cases}$$
Therefore,
 $$c_{r, s} y=\{f\delta - \rho^{-1}\sum_{j=f+2}^r\sum_{i=f+1}^j
g_{w_{ij}}^{-1} -\rho \sum_{j=f+2}^s\sum_{i=f+1}^j
g^*_{w_{ij}}\} e^f \n_\lambda,$$
where $w_{ij}=(i, j)$, the
transposition which switches $i$ and $j$. Now, the result follows
from the arguments similar to those for Hecke algebras in
\cite[3.32]{Ma}.
\end{proof}

In the remaining part of this section, unless otherwise stated, we
assume that $e>\max\{r, s\}$. Otherwise, $\H_{r}\otimes \H_s$ is not
semisimple over $\kappa$. So is $\mathscr B_{r, s}$.

Let $\lambda$ and $\mu$ be two partitions. We write $\lambda\supseteq \mu$ if $\lambda_i\ge \mu_i$ for all possible $i$'s.
Let $[\lambda/\mu]$ be the skew Young diagram obtained from
$[\lambda]$ by removing nodes in $[\mu]$.

For general cellular algebra $(A, \Lambda)$,
we use $[C(\lambda): D^\mu]$
to denote the multiplicity of irreducible $A$-module  $D^\mu$ in the cell module  $C(\lambda)$.

\begin{Lemma}\label{zero1} Suppose $(0, \lambda), (1, \mu)\in \Lambda_{r, s}$.
  If $ [C(0, \lambda): C(1, \mu)]\neq 0 $,
 then $\lambda^{(i)}\supset \mu^{(i)}$ and  $[\lambda^{(i)}/\mu^{(i)}]=\{p_i\}$ for some $p_i\in \mathscr R(\lambda^{(i)})$ such that
$\rho^2=q^{2k}$ and $k= \res(p_1)+\res(p_2)$.
\end{Lemma}
\begin{proof} We prove our result  by induction on $r+s$. The case
$r+s=2$ is trivial. In this case, $\lambda=((1), (1))$ and
$\mu=(\emptyset, \emptyset)$. In general, we assume $r\ge 2$.
Otherwise, we switch the role between $r$ and $s$ in the following arguments.

 Since we are assuming $e>\max\{r, s\}$, we have $C(0, \lambda)
=D^{0, \lambda}$ for any $\lambda\in \Lambda_{r,s}^0$. We consider
the restriction of  $C(0, \lambda)$ to $\mathscr B_{r-1, s}$. Note
that any  composition factor of $C(0, \lambda)$ is of form $C(0,
(\lambda^{(1)}\setminus \{p\}, \lambda^{(2)}))$ for some $p\in
\mathscr R(\lambda^{(1)})$. By Theorem~\ref{branching}, there is a
$p\in \mathscr R(\lambda^{(1)})$ such that either
$$[C(1, (\mu^{(1)}\setminus \{\tilde p\}, \mu^{(2)})): C(0,
(\lambda^{(1)}\setminus \{p\}, \lambda^{(2)}))]\neq 0, \text{ for
some $\tilde p\in \mathscr R(\mu^{(1)})$, }$$ or
 $$[C(0, (\mu^{(1)},
\mu^{(2)}\cup \{p_2\})):C(0, (\lambda^{(1)}\setminus \{p\},
\lambda^{(2)}))]\neq 0, \text{ for some $p_2\in \mathscr
A(\mu^{(2)})$}.$$ In the first case, by induction assumption,
$\lambda^{(1)}\setminus \{p\}\supset \mu^{(1)}\setminus \{\tilde
p\}$ and $\lambda^{(2)}\supset \mu^{(2)}$. Using Lemma~\ref{resi} yields $c(p)=c(\tilde p)$. Since we are
assuming $e>\max\{r, s\}$, we have $p=\tilde p$. So, $\lambda^{(i)}\supset \mu^{(i)}$  and      $[\lambda^{(i)}/\mu^{(i)}]=\{p_i\}$ for some $p_i,
i=1,2$. In the second case, since $\H_{r-1}\otimes \H_s$ is
semisimple, we have $( \lambda^{(1)}\setminus \{p\},
\lambda^{(2)})=(\mu^{(1)}, \mu^{(2)}\cup \{p_2\})$. We still have   $\lambda^{(i)}\supset \mu^{(i)}$ and
 $[\lambda^{(i)}/\mu^{(i)}]=\{p_i\}$ with $p_1=p$. Finally, using  Lemma~\ref{resi} yields   $\rho^2=q^{2k}$
with $k=\res(p_1)+\res(p_2)$, as required. \end{proof}

\begin{Lemma}\label{onearc} Suppose $\lambda=((r-1), \emptyset)$ with $r\ge 2$. We have
 $\det G_{1, \lambda}=0$ if and only if
$\rho^{2}\in \{q^{-2}, q^{2r-2}\}$.  \end{Lemma}
\begin{proof}
If $\det G_{1, \lambda}=0$, then $\Rad C(1, \lambda)\neq 0$. So,
there is an irreducible $\mathscr B_{r, 1}$-module, say $D^{\ell,
\mu}$ such that $D^{\ell, \mu}\subset \Rad C(1, \lambda)$ and
$(\ell, \mu)\lhd(1, \lambda)$. Therefore, $\ell\in \{0, 1\}$.

 We have
$\ell=0$. Otherwise, there is  a non-trivial homomorphism from $C(1,
\mu) $ to $C(1, \lambda)$. Applying
 the functor $\F$ to both  $C(1, \mu) $ and  $C(1, \lambda)$
 and using Lemma~\ref{functors}(b)  yields a non-trivial homomorphism from $C(0, \mu) $ to $C(0,
 \lambda)$, forcing $\lambda=\mu$, a contradiction. When $\ell=0$,
 by Lemma~\ref{zero1}, we have $\rho^{2}\in \{q^{-2}, q^{2r-2}\}$ if
 $\det G_{1, \lambda}=0$. Finally, we need to verify  $\det G_{1,
 \lambda}=0$ if $\rho^{2}\in \{q^{-2}, q^{2r-2}\}$.

 Let $v_i=g_{i,
 1} \n_\lambda e_1+\mathscr B_{r, 1}^{\rhd(1, \lambda)}\in C(1, \lambda)$.
Then the Gram matrix $G_{1, \lambda}$ is the $r\times r$ matrix
$(a_{ij})$ with entry $a_{i,j}=\langle v_i, v_j\rangle$. More
explicitly, $\det G_{1, \lambda}=0$ if and only if
$$
 \begin{vmatrix}
\rho^{-1}\delta & 1 &  -q^{-1} & \cdots & (-q)^{2-r}\\
1 & \rho^{-1}\delta+q-q^{-1} & (-q)^{-2} & \cdots & (-q)^{1-r}\\
-q^{-1} & (-q)^{-2}
&   \rho^{-1}\delta+q-q^{-3}   & \cdots & (-q)^{-r}\\
\cdot & \cdot & \cdot & \cdots & \cdot \\
 \cdot & \cdot & \cdot & \cdots & \cdot \\
(-q)^{2-r} & (-q)^{1-r}  & (-q)^{-r} & \cdots &   \rho^{-1}\delta+q-q^{3-2r}\\
\end{vmatrix}=0.$$
Let $\det (b_{ij})$ be the  left hand side of the above
equality. We have $b_{2,j}=-q^{-1} b_{1,j}$, $1\le
j\le r$ if $\rho^2=q^{-2}$. When $\rho^2=q^{2(r-1)}$,  $\sum_{j=1}^r
(-q)^{1-j} b_{i,j}=0$, $ 1\le i\le r$. In any  case, we have $\det
G_{1, \lambda}=0$.
\end{proof}

By similar arguments, we have the following result.

\begin{Lemma}\label{onearc1} Suppose
 $\lambda=((1^{r-1}), \emptyset)$ with $r\ge 2$.
Then $\det G_{1, \lambda}=0$ if and only if
$\rho^{2}\in \{q^{2}, q^{2-2r}\}$.  \end{Lemma}

\begin{Lemma}\label{ss-key1} If $\rho^2=q^{2(r+s-2)}$, then there is a $\mathscr B_{r, s}$-submodule
$M\subset C(1, \mu)$ with $\mu=((r-1), (s-1))$ such that $M\cong C(0, \lambda)$ with $\lambda=((r), (s))$.
\end{Lemma}
\begin{proof} We consider $\kappa$-space $M$ spanned by  $ v=\n_\lambda e_1 \n_\mu +\mathscr B_{r, s}^{\rhd (1, \mu)}\in C(1, \mu)$.
Since we are assuming that $e>\max\{r, s\}$, $v\neq 0$, forcing $M\cong   C(0, \lambda)$ as left $\mathscr H_r\otimes \mathscr H_{s}$-modules. It is routine to check that
$$e_1 v=(\delta-\rho \frac{1-q^{-2(r+s-2)}}{q-q^{-1}}) e_1 \n_\mu +\mathscr B_{r, s}^{\rhd (1, \mu)},$$
which is zero if $\rho^2=q^{2(r+s-2)}$. So,  $M\cong C(0, \lambda)$ as  $\mathscr B_{r, s}$-submodules. \end{proof}

The following result can be proved similarly. The only difference is that we have to use $\m_{\lambda'}$ instead of $\n_\lambda$
in the proof of Lemma~\ref{ss-key1}.

\begin{Lemma}\label{ss-key2} If $\rho^2=q^{-2(r+s-2)}$, then there is a $\mathscr B_{r, s}$-submodule
$M\subset C(1, \mu)$ with $\mu=((1^{r-1}), (1^{s-1}))$ such that $M\cong C(0, \lambda)$ with $\lambda=((1^r), (1^s))$.
\end{Lemma}

\begin{Prop}\label{semi1} Suppose $r, s\in \mathbb Z^{>0}$ and  $\delta\neq 0$. If $e>\max\{r, s\}$, then $\mathscr B_{r, s}$ is semisimple
if and only if  $ \prod_{\lambda}\det G_{1,
\lambda}\neq 0$ in $\kappa$ where
$\lambda\in \cup_{k=2}^{r+s-1}\{((k-1), \emptyset),
((1^{k-1}), \emptyset)\}$.
\end{Prop}
\begin{proof}  
We prove our result by induction on $r+s$ for all possible $r, s\in \mathbb Z^{>0}$.  We assume $r+s>2$ since $\mathscr B_{1,1}$ is semisimple.
When $r+s=3$, any cell module of $\mathscr B_{r, s}$ is of form
either $C(1, \lambda)$ or $C(0, \lambda)$. Note that $(r, s)=\{(2, 1),  (1, 2)\}$. By Lemma~\ref{onearc},
$\det G_{1, \lambda}\neq 0$ if and only if $\det G_{1, \mu}\neq 0$ where $\lambda=((1), (0))$ and  $\mu=((0), (1))$.
 By Theorem~\ref{cell-tool}(b), we have the result. In the remaining
part of the proof, we assume that $r+s\ge 4$.

``$\Leftarrow$" If $\mathscr B_{r, s}$ is not semisimple over
$\kappa$, by Theorem~\ref{cell-tool}(b), we have $\det G_{f,
\lambda}= 0$ for some $(f, \lambda)\in \Lambda_{r, s}$. Since we are
assuming $e>\max\{r, s\}$, $\H_r\otimes \H_s$ is semisimple. Note
that any cell module $C(0, \lambda)$ of $\mathscr B_{r, s}$ with
$(0, \lambda)\in \Lambda_{r, s}$ can be considered as the
corresponding cell module for $\H_r\otimes \H_s$. We have $\det G_{0,
\lambda}\neq 0$. Therefore, $f>0$. We prove our result by induction
on $r+s$.

We can find  a simple module, say $D^{\ell, \mu}$ such that
$D^{\ell, \mu}\subseteq \Rad C(f, \lambda)$. So, $(\ell, \mu)\lhd(f,
\lambda)$. In particular, $\ell\le f$. It results in a non-zero
homomorphism from $C(\ell, \mu)$ to $C(f, \lambda)$. We claim
$\ell=0$. Otherwise, applying the functor $\F$ to both $C(\ell,
\mu)$ and $C(f, \lambda)$ and using Lemma~\ref{functors}
 yields  a non-zero homomorphism from $C(0, \mu)$ to $C(f-\ell, \lambda)$.
So,  $\mathscr B_{r_1, s_1}$ is not semisimple where $r_1=r-\ell, s_1=s-\ell$. This contradicts our
induction assumption on $r+s-2\ell$. Now,  we  assume $\ell=0$ and
$f\ge 1$.

We can assume $r\ge s$. Otherwise, we switch the role between $r$ and $s$ in the following statements.  So, $r\ge 2$. We consider
$C(0, \mu)$ as $\mathscr B_{r-1, s}$-module.
 In this case, let $C(0, \nu)$ be a composition
factor of $C(0, \mu)$ for some $\nu$ with $|\nu|=|\mu|-1$.
By Theorem~\ref{branching}, $C(0, \nu)$ has to be a composition factor of a cell module, say $C(f_1, \alpha)$
  for $\mathscr B_{r-1, s}$. Further, $f_1\geq f-1\geq 0$.
 So, $\mathscr B_{r-1, s}$  is  not semisimple over $\kappa$ if $f>1$, a contradiction.
 When $f=1$, since we are assuming that $\delta\neq 0$, which is equivalent to $\rho^2\neq 1$,
  by Lemmas~\ref{zero1}-\ref{onearc1}, there is a  $\lambda\in \cup_{k=2}^{r+s-1}\{((k-1), \emptyset),
((1^{k-1}), \emptyset)\}$, such that $\det G_{1,
\lambda}=0 $, a contradiction.

\medskip

 "$\Rightarrow$" Suppose $ \det G_{1,
\lambda}=0$, for some  $\lambda\in \{((k-1),
\emptyset), ((1^{k-1}), \emptyset)\}$ in $\kappa$.  We claim that $k\in \{r+s-1, r+s-2\}$.  Otherwise, by induction on $r+s-3$,
 $\mathscr B_{r-1, s-1}$ is not semisimple. Therefore,
  $\det G_{\ell, \mu}=0$ for some $(\ell, \mu)\in \Lambda_{r-1, s-1}$. Applying  the functor $\G$ to
   the cell module $C(\ell, \mu)$ yields  $\det G_{\ell+1, \mu}=0$. This contradicts our assumption that $\mathscr B_{r, s}$ is semisimple.
 Therefore,
$k\in \{r+s-1, r+s-2\}$. By Lemmas~\ref{onearc}-\ref{onearc1},  $\rho^2\in \{ q^{\pm 2}, q^{\pm 2(r+s-2)},  q^{\pm 2(r+s-3)}\}$.

 In fact,  $\rho^2\neq q^{\pm2}$. Otherwise,  $\mathscr B_{r-1, s-1}$ is not semisimple.  So, $\mathscr B_{r, s}$ is not semisimple, either.
 If   $\rho^2\neq   q^{\pm 2(r+s-2)}$, by Lemmas~\ref{ss-key1}-\ref{ss-key2},  $\mathscr B_{r, s}$ is not semisimple.

Suppose $\rho^2=  q^{ 2(r+s-3)}$. We can assume  $r+s> 4$. Otherwise, $\rho^2=q^2$, which has already been discussed. By Remark~\ref{rbigs}, we can assume
$r\ge 3$.

 By Lemma~\ref{ss-key1}, $\Hom (C(0, \lambda), C(1, \mu))\neq 0$ where
$\lambda=((r-1), (s))$ and $\mu= ((r-2), (s-1))$. By Theorem~\ref{branching}, $\Hom (C(0, \lambda),   \Res^L C(1, \nu))\neq 0$
 with $\nu=((r-2, 1), (s-1))$. By  Frobenius reciprocity, $\Hom( \Ind^L C(0, \lambda),  C(1, \nu))\neq 0$. By Proposition~\ref{Gind},
and Theorem~\ref{branching}, there is a filtration
$$0\subset C(1,(r-1), (s-1))\subset M \subset \Ind^L C(0, \lambda)$$  such that
$$M/ C(1, (r-1), (s-1))\cong C(0, (r), (s)) $$ and $$
\Ind^L C(0, \lambda)/M\cong  C(0, ((r-1,1), (s))) .$$
By Lemma~\ref{zero1}  $\Hom(N, C(1, \nu))=0$ if $N=C(0, (r), (s))$. Since we are assuming that   $e>\max\{r, s\}$,
$\Hom(N, C(1, \nu))=0$ if $N=C(1, (r-1), (s-1))$. So,
 $$\Hom( C(0, (r-1, 1), (s-1)),  C(1, \nu))\neq 0, $$ forcing
 $\mathscr B_{r, s}$ being  non semisimple, a contradiction. So, $\rho^2\neq  q^{ 2(r+s-3)}$. Finally, we remark that
 we can  prove  $\rho^2\neq  q^{- 2(r+s-3)}$ by arguments similar to those as above. The only difference is that we have to use conjugates of $\gamma$ instead of bipartition $\gamma$ in the previous
 arguments.
 We leave the details to the reader.
\end{proof}

\begin{Theorem}\label{semimain} Suppose  $r, s\in \mathbb Z^{>0}$.  Then $\mathscr B_{r, s}$ is (split)
semisimple over $\kappa$  if and only if $e>\max\{r, s\}$ and one of
the following conditions holds:
\begin{enumerate} \item  $\delta\neq 0$ and   $\rho^2\neq
q^{2a}$ for any $a\in \mathbb Z$ with $|a|\le r+s-2$.
\item   $\delta= 0$ and  $(r, s)\in \{(1,2), (2,1), (1, 3), (3,
1)\}$.
\end{enumerate}
\end{Theorem}
\begin{proof} Since $\H_r\otimes \H_s\cong \mathscr B_{r, s}/\langle e_1\rangle$,
where $\langle e_1\rangle $ is the two-sided ideal of $\mathscr
B_{r, s}$ generated by $e_1$, $\mathscr B_{r, s}$ is not semisimple
if  $e\le \max\{r, s\}$. So, we can assume $e>\max\{r, s\}$ when we
discuss the semisimplicity of $\mathscr B_{r, s}$. We remark that
the  result for $\delta\neq 0$  follows from
Theorem~\ref{cell-tool}(b),
 Lemmas~\ref{onearc}--\ref{onearc1} and Proposition~\ref{semi1}.

 Suppose $\delta=0$. By Theorem~\ref{main1}, $\mathscr B_{r, r}$ is not semisimple for all $r\in \mathbb Z^{>0}$.
 In the remaining part of the proof, we assume $r> s$ since $\mathscr B_{r,s}\cong \mathscr B_{s,r}$.

 When $r+s<5$, either  $r=2, s=1$ or $r=3$ and $s=1$. By Lemmas~\ref{onearc}--\ref{onearc1},
 $\det G_{1, \lambda}\neq 0$ for $\lambda\in \{((1), \emptyset), ((2),
 \emptyset), ((1^2), \emptyset)\}$. So, both  $\mathscr B_{2, 1}$
 and $\mathscr B_{3, 1}$ is semisimple over $\kappa$.
Now, we assume $r+s\ge 5$.

It is easy to see that the Gram matrix $G_{1, \lambda}$  for $\lambda\in \{((2),(1)), ((2,1),\emptyset), ((1^3), (1))\}$ is $a\times a$ matrix
with $a\in \{6, 8\}$. We use MATLAB software  to check $\det G_{1, \lambda}=0$.
So, $\mathscr B_{r,s}$ is not semisimple if $r+s=5$. Further,
 $\Rad C(1, \lambda)$ contains a non-zero irreducible module, say $D^{\ell, \mu}$, such that $(\ell, \mu)<(1, \lambda)$.
We claim $\ell=0$. Otherwise, Applying  the functor $\F$ to  both
$C(1, \lambda)$ and $C(1, \mu)$ yields a non-zero homomorphism from
$C(0, \mu)$ to $C(0, \lambda)$. Since we are assuming
$o(q^2)>\max\{r, s\}$, both  $C(0, \mu)$ and $C(0, \lambda)$ are
irreducible. Therefore, $\lambda=\mu$, a contradiction. So,
$\ell=0$. Applying the functor $G$ to  both $C(0, \mu)$ and $C(1,
\lambda)$, repeatedly yields  a non-zero homomorphism from $C(k,
\mu)$ to $C(1+k, \lambda)$ with $2k+|\mu|=r+s$. Since  $(k,
\mu)<(1+k, \lambda)$, $\det G_{1+k, \lambda}=0$. By
Theorem~\ref{cell-tool}(b), $\mathscr B_{r, s}$ is not semisimple if
$r+s\ge 5$ and $r=s+\ell$ with $\ell\in \{1, 2, 3\}$.

Finally, we assume  $r=s+(k+2)$ with $k\ge 2$ and $r+s> 5$. We
claim $\det G_{1, \lambda}=0$ if $\lambda=((2, 1^{k}), \emptyset)$
for $k\ge 1$. If so, standard arguments  on the functor $\G$ shows
that  $ \det G_{s, \lambda}=0$ and hence, by
Theorem~\ref{cell-tool}(b), $\mathscr B_{r, s}$ is not semisimple.

Suppose $\det G_{1, \lambda}\neq 0$. Then $C(1, \lambda)=D^{1,
\lambda}$. We have already verified $\det G_{1, ((2,1),
\emptyset)}=0$. In general, by Theorem~\ref{branching}, $C(1, \mu)$
is a submodule of $\Res^LC (1, \lambda)$ where $\mu=((2, 1^{k-1}),
\emptyset)$. By induction assumption, we have $\det G_{1, \mu}=0$.
So, $\Rad C(1, \mu)$ contains an irreducible $ {\mathscr B}_{r-1,
s}$--module $D^{f, \nu}$ with $(f, \nu)<(1, \mu)$.  We have $f\neq
1$. Otherwise,   applying
 the functor $\F$ to  both $C(1,
\mu)$ and $C(1, \nu)$ yields  $\nu=\mu$, a contradiction. So,
$f=0$. Since we are assuming  $\rho^2=1$,  by Lemma~\ref{zero1},
$\nu=((2^2, 1^{k-2}), (1))$. By Frobenius reciprocity, $\Hom(\Ind^L C(0, \nu), C(1,
\lambda)\neq 0$.  Using Proposition~\ref{Gind}  yields $\Ind^L
C(0, \nu)=\Res^R C(1, \nu)$. By Theorem~\ref{branching},
this module  has a filtration of cell modules such that each section
is of form either $C(1, ((2^2, 1^{k-2}), \emptyset))$ or
$C(0, \gamma)$'s for some $\gamma$, where each  $\gamma$ can be obtained from $(2^2, 1^{k-2})$ by
adding an addable node. So, $C(1, \lambda)=D^{1, \lambda}$ has to be a composition factor of
either $C(1, ((2^2, 1^{k-2}), \emptyset))$ or $C(0,
\gamma)$, forcing $\lambda=((2^2, 1^{k-2}), \emptyset)$, a contradiction.  So, our claim follows.
\end{proof}


\providecommand{\bysame}{\leavevmode ---\ } \providecommand{\og}{``}
\providecommand{\fg}{''} \providecommand{\smfandname}{and}
\providecommand{\smfedsname}{\'eds.}
\providecommand{\smfedname}{\'ed.}
\providecommand{\smfmastersthesisname}{M\'emoire}
\providecommand{\smfphdthesisname}{Th\`ese}

\end{document}